\documentclass[11pt]{article}
\usepackage{amsmath,mathrsfs}
\usepackage{amssymb}
\usepackage{amsthm}
\usepackage{indentfirst}
\usepackage{color}

\allowdisplaybreaks

\newtheorem{theorem}{Theorem}[section]
\newtheorem*{theorem*}{Theorem}

\newtheorem{remark}[theorem]{Remark}
\newtheorem{lemma}[theorem]{Lemma}
\newtheorem{definition}[theorem]{Definition}

\newtheorem{proposition}[theorem]{Proposition}
\newtheorem{example}[theorem]{Example}
\newtheorem*{proposition*}{Proposition}

\newcommand{\R}{\mathbb{R}}

\newcommand{\E}{\mathbb{E}}

\newcommand{\Dxp}{D_{\bx_p}}

\newcommand{\be}{\begin{eqnarray*}}
\newcommand{\ee}{\end{eqnarray*}}
\newcommand{\ba}{\begin{align*}}
\newcommand{\bpm}{\begin{pmatrix}}
\newcommand{\epm}{\end{pmatrix}}

\newcommand{\bx}{\boldsymbol{x}}
\newcommand{\bn}{\boldsymbol{n}}

\newcommand{\by}{\boldsymbol{y}}

\begin{document}

\title{Generalized partial-slice monogenic functions: \\
the octonionic case}

\author{Zhenghua Xu$^1$\thanks{This work was partially supported by  Anhui Provincial Natural Science Foundation (No. 2308085MA04), Fundamental Research Funds for the Central Universities (No. JZ2025HGTG0250) and China Scholarship Council (No. 202506690052).}
\ and Irene Sabadini$^2$\thanks{This work was partially supported by PRIN 2022 {\em Real and Complex Manifolds: Geometry and Holomorphic Dynamics}.}\\
\emph{$^1$\small School of Mathematics, Hefei University of Technology,}
\emph{\small  Hefei, 230601, P.R. China}\\
\emph{\small E-mail address: zhxu@hfut.edu.cn}
\\
\emph{\small $^2$Dipartimento di Matematica, Politecnico di Milano,}
\emph{\small Via E. Bonardi, 9, 20133 Milano, Italy} \\
\emph{\small E-mail address:   irene.sabadini@polimi.it}
}

\date{}

\maketitle

\begin{abstract}
In a recent paper [Trans. Amer. Math. Soc. 378 (2025),   851-883],   the concept of generalized partial-slice monogenic (or regular) function was introduced  over Clifford algebras.  The present paper shall extend the study of generalized partial-slice monogenic  functions from the associative  case of Clifford algebras to non-associative alternative algebras, such as octonions. The new class of functions encompasses the regular functions [Rend. Sem. Mat. Univ. Padova 50 (1973),   251-267] and slice regular functions [Rocky Mountain J. Math. 40 (2010), no. 1,  225-241] over  octonions, indeed both appear in the theory  as special cases.  In the non-associative setting of octonions,  we shall develop some fundamental  properties such as  identity theorem,   Representation Formula,  Cauchy (and Cauchy-Pompeiu) integral formula,  maximum modulus principle, Fueter polynomials, Taylor series expansion. As a complement, the paper also introduces and discusses the notion of generalized partial-slice (and regular) functions. Although the study is limited to the case of octonions, it is clear from the statements and the arguments in the proofs that the results hold more in general in real alternative algebras equipped with a notion of conjugation.
 \end{abstract}
{\bf Keywords:}\quad Functions of a hypercomplex variable;  monogenic functions; slice monogenic functions;  alternative algebras, octonions\\
{\bf MSC (2020):}\quad  Primary: 30G35;  Secondary: 32A30,   17A35

\section{Introduction}
As a generalization of holomorphic functions of one complex variable, in 1973 Dentoni and Sce \cite{Dentoni-Sce} (see the English translation in  \cite{Colombo-Sabadini-Struppa-20}) generalized the notion of \textit{(Fueter) regular function} from quaternions to octonions   (also called Cayley numbers). Thereafter, the interest in  regular functions  over octonions continued in various works among which we mention \cite{Colombo-Sabadini-Struppa-00,Li-Peng-01,Li-Peng-02,Nono}. {Inspired by the idea in \cite{Gentili-Struppa-07},} Gentili and Struppa \cite{Gentili-Struppa-10} in 2010 introduced the  definition of regularity (now called \textit{slice regularity}) on the space of Cayley numbers.
Without claiming completeness, the reader may consult \cite{Colombo-Krausshar-Sabadini-24,Ghiloni-Perotti,Ghiloni-Perotti-Stoppato-20,Perotti-22,Wang-17,Xu-21} and references therein for more results  on octonionic slice-regular functions and \cite{Dou-23,Dou,Jin-Ren-20,Jin-Ren-21,Jin-Ren-Sabadini-20} for other variations of slice  analysis over octonions.

This paper is a continuation of the work \cite{Xu-Sabadini} where  the
concept of \textit{generalized partial-slice monogenic functions}  was  introduced for functions with values in a Clifford algebra and immediately elaborated, in early 2023, in the subsequent paper \cite{DingXu}. The function theory includes the two theories of monogenic functions \cite{Brackx} and of slice monogenic functions \cite{Colombo-Sabadini-Struppa-09, Colombo-Sabadini-Struppa-11}, respectively.  See   \cite{Huo-24,Xu-Sabadini-2,Xu-Sabadini-3,Xu-Sabadini-4} for more results in this new setting. As a special case of alternative  algebras, it is well known that all Clifford algebras  are   associative, so it is natural to  further investigate generalized partial-slice monogenic functions  in the  setting of a non-associative alternative $\ast$-algebras, in particular in the case of octonions.

The notion of generalized partial-slice monogenic functions over octonions  encompasses the regular functions studied by Dentoni and Sce \cite{Dentoni-Sce}, slice regular functions in the sense of Gentili and Struppa \cite{Gentili-Struppa-10} and also functions in the kernel of the slice Dirac operator in octonions by Jin and Ren \cite{Jin-Ren-20}.
The idea of partial-sliceness has been developed in the theory of $T$-regular functions, which  proposes a unified theory of regularity in one hypercomplex variable for an alternative $\ast$-algebra, see \cite{Ghiloni-Stoppato-24-1,Ghiloni-Stoppato-24-2}, where all results, except the representation formula (and its consequences),  are obtained for associative algebras.

The main novelty of this paper is then the consideration of the non-associative case which is treated in detail in the case of octonions.
In slice analysis, a very useful tool is the so-called splitting lemma, which implies that slice regular functions inherit various properties  of the holomorphic functions that appear in the splitting. Unlike all cases of slice regular or slice monogenic functions, the splitting lemma no longer holds for
generalized partial-slice monogenic functions over octonions and, more in general, in the non-associative case.  This peculiarity entails that various results that in slice analysis may be deduced using the splitting lemma need to be obtained using different arguments in our octonionic case. The non-associativity emerges in various statements and proofs in the paper, and is particularly relevant in the study of generalized partial-slice monogenic homogeneous polynomials which are the building blocks of the Taylor series expansion.

The  numerous results that we obtain in the paper are briefly outlined   as follows.
In Section 2,  we recall some basic definitions about real alternative algebras
and octonions,  and then recall the notions of regular functions and of slice regular
functions in the octonionic framework.

In Section 3,  we shall set up some basic results  for generalized partial-slice monogenic functions over octonions. We define the so-called \textit{slice domains} and  \textit{partially symmetric domains}, which are useful to prove results such as the identity theorem  which is then  used  to  establish the Representation Formula (Theorem \ref{Representation-Formula-SM}), a fundamental tool in slice analysis.

In  Section  4, we establish a Cauchy-Pompeiu integral formula (Theorem \ref{CauchyPompeiuslice}) in the non-associative case of octonions, which gives a slice version of the Cauchy integral formula (Theorem \ref{Cauchy-slice}) as well as some  consequences, such as mean value theorem and maximum modulus principle.

 In  Section 5,  we  formulate the notion of  generalized partial-slice  (and regular) functions over octonions and  prove a global   version of Cauchy-Pompeiu integral  formula (Theorem \ref{Cauchy-Pompeiu}), where the Cauchy kernel turns out to be operator-valued. In addition,  we  obtain the relation between the set of generalized partial-slice monogenic  and regular functions over octonions, see Theorem \ref{relation-GSR-GSM}.

 Finally, in  Section  6, we  study the  Taylor series expansion for generalized partial-slice monogenic   functions over octonions. To this end, we  construct generalized partial-slice monogenic Fueter-type polynomials over octonions, and, as a tool, we discuss a Cauchy-Kovalevskaya extension starting from some real analytic functions defined in some  domains of $\mathbb{R}^{p+1}$, see Definition \ref{Cauchy-Kovalevska-extension}.

\section{Preliminaries}
In this section, we collect some preliminary results on alternative algebras and octonions and recall   the notions of    regular  functions and  of slice  regular  functions.
\subsection{Real alternative algebras}
\noindent Let $\mathbb{A}$ be a real algebra with a unity. Recall that a real algebra $\mathbb{A}$ is said to be alternative if the \textit{associator}
$$[a,b,c]:= (ab)c-a(bc)$$ is a trilinear, alternating function in  the variables $a,b,c\in \mathbb{A}$.

 Two important results that allow to work in an alternative algebra $\mathbb{A}$ are the following:
\begin{itemize}
\item
The Artin theorem asserting that
\textit{the subalgebra generated by two elements of $\mathbb  A$ is associative}.
\item
The Moufang identities:
$$a(b(ac)) = (aba)c,\qquad ((ab)c)b = a(bcb), \qquad (ab)(ca) = a(bc)a,$$
for  $a,b,c\in\mathbb A$.
\end{itemize}
A special case of real alternative algebra with unit   is the algebra of octonions, described in detail in the next subsection.

\subsection{The algebra  of octonions}
Let $\mathcal{B}=\{e_0=1, e_1, e_2, \ldots, e_7 \}$   be  the standard orthogonal basis    of $\mathbb{R}^{8}$ and let
$$\Xi=\{(1,2,3), (1,4,5),(2,4,6),(3,4,7),(5,3,6), (6,1,7),(7,2,5)\}.$$
A possible construction of the real algebra of  octonions is done by taking the span of $\mathcal{B}$,   endowed with the multiplication rule
$$e_ie_j=-\delta_{ij}+\epsilon_{ijk}e_k, \ i,j,k\in \{1,2,\ldots,7\},$$
where $\delta_{ij}$ is the Kronecker symbol and
\begin{eqnarray*}
\epsilon_{ijk}=
\begin{cases}
(-1)^{\tau (\sigma)}, \qquad   (i,j,k)\in \sigma(\Xi)=\{ \sigma(\xi): \xi\in \Xi \},
\\
0\qquad \qquad \qquad \mbox {otherwise},
\end{cases}
\end{eqnarray*}
here $\sigma$  denotes a permutation and  $\tau (\sigma)$ is its sign.

Any $x\in \mathbb{O}$ can be expressed as
$$x=x_0+\sum_{i=1}^{7}x_ie_i, \quad x_i\in \mathbb{R}, $$
its conjugate is defined as $$\overline{x}=x_0-\sum_{i=1}^{7}x_ie_i, $$
and the modulus of $x$ is defined as $|x|=\sqrt{x\overline{x}}$, which is exactly the usual Euclidean norm in $\mathbb{R}^{8}$.
Furthermore, the modulus is multiplicative, i.e., $|xy|=|x||y|$ for all $x,y \in \mathbb O$. Every nonzero   $x\in\mathbb{O}$ has a multiplicative  \textit{inverse} given by $x^{-1}=|x|^{-2}\overline{x}$. Hence,  $\mathbb{O}$ is a  non-commutative, non-associative, normed, and division algebra.    See for instance \cite{Baez,Okubo,Schafer}  for more explanations on  alternative algebras and   octonions.

In this paper, we shall make use of the following useful property which is an immediate consequence of the Artin  theorem.
\begin{proposition}\label{artin}
For any $x,y \in \mathbb{O},$ it holds that
$$[x,x,y]=[\overline{x},x,y]=0.$$
\end{proposition}

\subsection{Regular and slice regular functions}
In this  subsection,  we  recall the definitions of  regular and of slice regular  functions.

Throughout this paper,  an element $(x_0,x_1,\ldots,x_7)\in\R^{8}$ will be identified with the octonion $x\in \mathbb{O}$ via
$$(x_0,x_1,\ldots,x_7) \rightarrow x=x_0+\sum_{i=1}^{7}x_ie_i. $$ We consider the functions  $f:\Omega\longrightarrow  \mathbb{O}$, where $\Omega\subseteq \mathbb{O}$ is a domain (i.e., connected open set). As usual, denote $\mathbb{N}=\{0,1,2,\ldots\}$.  For $k\in \mathbb{N}\cup \{\infty\}$,  denote by $ C^{k}(\Omega,\mathbb{O})$ the set of all functions  $f (x)=\sum_{i=0}^{7}  e_{i}   f_{i}(x)$ with real-valued $f_{i}(x)\in  C^{k}(\Omega)$.

\begin{definition}[Regular    function \cite{Dentoni-Sce,Colombo-Sabadini-Struppa-20}] \label{monogenic-Clifford}
Let $\Omega $ be a domain   in $\mathbb{O}$ and let  $f\in C^{1}(\Omega,\mathbb{O})$.  The function $f (x)=\sum_{i=0}^{7}  e_{i}   f_{i}(x)$ is called  left regular (or  left $\mathbb{O}$-analytic) in $\Omega $ if it satisfies the   equation
$$ D_{x}f(x):=\sum _{i=0}^{7}e_{i} \frac{\partial f}{\partial x_{i}}(x)
= \sum _{i,j=0}^{7}   e_{i}e_{j} \frac{\partial f_{j}}{\partial x_{i}}(x)=0, \quad x\in \Omega. $$
Similarly, the function $f$ is called  right   regular (or right $\mathbb{O}$-analytic) in $\Omega $ if
$$ f(x)D_x:=\sum _{i=0}^{7} \frac{\partial f}{\partial x_{i}}(x)e_{i}= \sum _{i,j=0}^{7}   e_{j} e_{i} \frac{\partial f_{j}}{\partial x_{i}}(x)=0, \quad x\in \Omega. $$
\end{definition}
The operator in Definition  \ref{monogenic-Clifford}
$$D_x=\sum_{i=0}^{7}e_i\frac{\partial}{\partial x_i}=\sum_{i=0}^{7}e_i\partial_{x_i}$$
is called the   generalized Cauchy-Riemann operator   in octonionic analysis.

The set of square roots of $-1$ in $\mathbb{O}$ is the $6$-dim unit sphere given  by
 $$S^{6}= \{I \in \mathbb{O} \mid I^{2} = -1 \}.$$
For  each  $I \in S^{6}$, denote by $\mathbb{C}_{I}:=\langle1, I\rangle \cong \mathbb{C}$ the subalgebra of $\mathbb{O}$ generated by $1$ and $I$. Notice that each  $x\in \mathbb O$ can be expressed as $x=x_0+r I$ with  $x_0\in \mathbb R, r\geq0$ and $I\in  S^{6}$. This observation allows decomposing $\mathbb{O}$ into `complex slices'
$$\mathbb{ O}=\bigcup_{I \in  S^{6}} \mathbb{C}_{I},$$
which  derives the   notion of slice regularity over octonions  \cite{Gentili-Struppa-10}.
\begin{definition}[Slice regular function]\label{slice-regular}
 Let $\Omega$ be a domain in $\mathbb{O}$. A function $f :\Omega \rightarrow \mathbb{O}$ is called (left)  slice regular if, for all $I \in   S^{6}$, its restriction $f_{I}$ to $\Omega_{I}=\Omega\cap  \mathbb{C}_I   \subseteq \mathbb{R}^{2}$ is (left) holomorphic, i.e.,  $f_{I}\in C^{1}(\Omega_{I},\mathbb{O})$ satisfies
$$ \partial_{x_0}f_{I}(x_0+rI) +I\partial_{r}  f_{I}(x_0+rI)=0, \quad x_0+rI \in \Omega_{I}.$$
Similarly, the function $f$ is called  right slice  regular  if, for all $I \in   S^{6}$, its restriction $f_{I} \in C^{1}(\Omega_{I},\mathbb{O})$ satisfies
$$ \partial_{x_0}f_{I}(x_0+rI) +\partial_{r}  f_{I}(x_0+rI)I=0, \quad x_0+rI \in \Omega_{I}.$$
 \end{definition}

\section{Generalized partial-slice monogenic functions}
In the sequel, let $p\in  \mathbb{N}$, $0\leq p\leq 6$, and set $q=7-p$. From now on, we shall split the element $x\in \mathbb{O}$ into
$$\bx=\bx_p+\underline{\bx}_q \in\R^{p+1}\oplus\R^{q}, \quad \bx_p=\sum_{i=0}^{p}x_i e_i,\ \underline{\bx}_q=\sum_{i=p+1}^{7}x_i e_i.$$
Here  we write the  element $x$ as $\bx$ in bold to emphasize the splitting.

Similarly, the generalized Cauchy-Riemann operator   and the Euler operator are split as
\begin{equation}\label{Dxx}
D_{\bx}=\sum_{i=0}^{7}e_i\partial_{x_i}=\sum_{i=0}^{p}e_i\partial_{x_i}+
\sum_{i=p+1}^{7}e_i\partial_{x_i}=:D_{\bx_p}+D_{\underline{\bx}_q},
\end{equation}
\begin{equation*}\label{Exx}
\E_{\bx}=\sum_{i=0}^{7}x_i\partial_{x_i}=\sum_{i=0}^{p}x_i\partial_{x_i}+
\sum_{i=p+1}^{7}x_i\partial_{x_i}=:\E_{\bx_p}+\E_{\underline{\bx}_q}.
\end{equation*}

Denote by $\mathbb{S}$ the unit sphere   in $\mathbb R^q$, whose elements $(x_{p+1},\ldots, x_{7})$ are identified with $\underline{\bx}_q=\sum_{i=p+1}^{7}x_i e_i$, i.e.
$$\mathbb{S}=\big\{\underline{\bx}_q: \underline{\bx}_q^2 =-1\big\}=\big\{\underline{\bx}_q=\sum_{i=p+1}^{7}x_i e_i:\sum_{i=p+1}^{7}x_i^{2}=1\big\}.$$
Note that, for $\underline{\bx}_q\neq0$, there exists a uniquely determined $r\in \mathbb{R}^{+}=\{x\in \mathbb{R}: x>0\}$ and $\underline{\omega}\in \mathbb{S}$, such that $\underline{\bx}_q=r\underline{\omega}$, more precisely
 $$r=|\underline{\bx}_q|, \quad \underline{\omega}=\frac{\underline{\bx}_q}{|\underline{\bx}_q|}. $$
 When $\underline{\bx}_q= 0$ we set $r=0$ and $\underline{\omega}$ is not uniquely defined since   $\bx=\bx_p+0 \underline{\omega}  $ for all $\underline{\omega}\in \mathbb{S}$. Hence, for any $\bx=\bx_p+\underline{\bx}_q \in \mathbb{O}$, we set $\bx':=(\bx_p,r)=(x_0,x_1,\ldots, x_p,r)\in \mathbb{R}^{p+2}$ with $r=|\underline{\bx}_q|$.

The upper half-space $\mathrm{H}_{\underline{\omega}}$ in $\mathbb{R}^{p+2}$ associated with $\underline{\omega}\in \mathbb{S}$ is defined by
$$\mathrm{H}_{\underline{\omega}}=\{\bx_p+r\underline{\omega}, \bx_p \in\R^{p+1}, r\geq0 \},$$
and it is clear from the previous discussion that
$$ \mathbb{O}=\bigcup_{\underline{\omega}\in \mathbb{S}} \mathrm{H}_{\underline{\omega}},$$
and
$$ \R^{p+1}=\bigcap_{\underline{\omega}\in \mathbb{S}} \mathrm{H}_{\underline{\omega}}.$$

In the sequel, we shall make use of the notation
$$\Omega_{\underline{\omega}}:=\Omega\cap (\mathbb{R}^{p+1} \oplus \underline{\omega} \mathbb{R})\subseteq \mathbb{R}^{p+2},$$
where $\Omega$ is an open set in $\mathbb{O}$.

Recalling the notation in formula \eqref{Dxx},  we now introduce the  definition of  generalized partial-slice monogenic functions over octonions as follows.
\begin{definition} \label{definition-slice-monogenic}
 Let $\Omega$ be a domain in $\mathbb{O}$. A function $f :\Omega \rightarrow \mathbb{O}$ is called left  generalized partial-slice monogenic of type $p$ if, for all $ \underline{\omega} \in \mathbb S$, its restriction $f_{\underline{\omega}}$ to $\Omega_{\underline{\omega}}\subseteq \mathbb{R}^{p+2}$  is of class $C^{1}$ and  satisfies
$$D_{\underline{\omega}}f_{\underline{\omega}}(\bx):=(D_{\bx_p}+\underline{\omega}\partial_{r}) f_{\underline{\omega}}(\bx_p+r\underline{\omega})=0,$$
for all $\bx=\bx_p+r\underline{\omega} \in \Omega_{\underline{\omega}}$.
\\
We denote by $\mathcal {GSM}(\Omega)$ (or $\mathcal {GSM}^{L}(\Omega)$ when needed)  the class of  all  left generalized partial-slice monogenic functions of type $p$ in $\Omega$.
\\
Similarly,  denote by $\mathcal {GSM}^{R}(\Omega)$ the set of  all right  generalized partial-slice monogenic functions of type $p$,  which are defined by the condition
$$f_{\underline{\omega}}(\bx)D_{\underline{\omega}}:={f_{\underline{\omega}} (\bx_p+r\underline{\omega})D_{\bx_p}}+ \partial_{r}f_{\underline{\omega}} (\bx_p+r\underline{\omega})\underline{\omega}=0,$$
for all $\bx=\bx_p+r\underline{\omega} \in \Omega_{\underline{\omega}}$.
 \end{definition}
 Unlike the Clifford algebra case,  we remark that    $\mathcal {GSM}^{L}(\Omega)$ is not a right $\mathbb{O}$-module due to the non-associativity of octonions; this is a well known fact also for regular and for slice regular functions of the octonionic variable, and for the specific case we are treating the reader may see Example \ref{za} below.
\begin{remark}\label{rem32}
{\rm
 When $p=6$, the  notion of generalized partial-slice monogenic functions in Definition  \ref{definition-slice-monogenic} coincides with the notion of   regular  functions  in Definition  \ref{monogenic-Clifford}.

When $p=0$, Definition  \ref{definition-slice-monogenic}  gives   one of the slice regular functions in Definition \ref{slice-regular}.}
\end{remark}
In the sequel we shall omit to specify the type $p\in\{0,1,\ldots,6\}$  unless specifically stated, which shall be fixed from the context.

\begin{remark} {\rm
The operator $D_{\underline{\omega}}$ in Definition  \ref{definition-slice-monogenic} is given in the sense of the Gateaux derivative. In fact, it can be rewritten   in the classical Fr\'{e}chet sense as   a global  nonconstant coefficients differential operator
 $$\overline{\vartheta} f(\boldsymbol{x})=\Dxp f(\bx)+\frac{\underline{\bx}_q}{|\underline{\bx}_q|^2}\mathbb{E}_{\underline{\boldsymbol{x}}_{q}} f(\bx),$$
where  $ f \in C^{1}(\Omega,\mathbb{O})$ and $\Omega$ is a domain in $\mathbb{O}\setminus \mathbb{R}^{p+1}$.}
 \end{remark}

Before discussing more properties of the function class $\mathcal {GSM}(\Omega)$,  we  provide some examples to illustrate the variety of functions belonging to this class.
\begin{example}\label{example-1}
Let $p\in\{0,1,\ldots,5\}$ and $\underline{\omega} \in\mathbb{S}$.  With the above notations,  let $D=\mathrm{H}_{\underline{\omega}}\cup \mathrm{H}_{-\underline{\omega}}$ and set
\begin{eqnarray*}
f(\bx)=
\left\{
\begin{array}{ll}
0,     &\mathrm {if} \ \bx \in  \mathbb{O} \setminus D,
\\
1,   &\mathrm {if} \ \bx \in D \setminus \mathbb{R}^{p+1}.
\end{array}
\right.
\end{eqnarray*}
 Then $f \in \mathcal {GSM}^{L}(\mathbb{O}\setminus \mathbb{R}^{p+1}) \cap{\mathcal{GSM}}^{R}(\mathbb{O} \setminus \mathbb{R}^{p+1}).$
\end{example}
For $p=6,$ the function  $f$ given by Example \ref{example-1} should be interpreted  as   constant since $D=\mathrm{H}_{\underline{\omega}}\cup \mathrm{H}_{-\underline{\omega}}=\mathbb{O}$.
\begin{example}Given $n\in \mathbb{N}$, set
$$f(\bx)=(x_{0}+\underline{\bx}_{q})^{n}, \quad \bx=  \sum_{i=0}^{p}x_i e_i + \underline{\bx}_q, \ \underline{\bx}_q=\sum_{i=p+1}^{7}x_i e_i.$$
Then $f \in \mathcal {GSM}^{L}(\mathbb{O}) \cap{\mathcal{GSM}}^{R}(\mathbb{O}).$
\end{example}
\begin{example}\label{Cauchy-kernel-example}
Consider the Cauchy kernel
$$E(\bx):=\frac{1}{\sigma_{p+1}}\frac{\overline{\bx}}{|\bx|^{p+2}}, \quad \bx \in\Omega=\mathbb{O}\setminus \{0\},$$
where $\sigma_{p+1}=2\frac{\Gamma^{p+2}(\frac{1}{2}) }{\Gamma (\frac{p+2}{2})} $ is the surface area of the unit ball in $\mathbb{R}^{p+2}$. \\
Then $E(\bx) \in \mathcal {GSM}^{L}(\Omega) \cap{\mathcal{GSM}}^{R}(\Omega).$
\end{example}

\begin{remark}\label{C1-C00}{\rm
A function $f\in \mathcal {GSM}(\Omega)$ is such that $f_{\underline{\omega} }\in C^{\infty}(\Omega_{\underline{\omega}}, \mathbb{O})$ for all $\underline{\omega}\in \mathbb S$. This is a consequence of Theorem \ref{Cauchy-slice} whose proof is postponed to Section \ref{Sect4} but which is independent of what follows in the rest of this section.}  \end{remark}

For $\mathrm{k}=(k_0,k_1,\ldots,k_{p})\in\mathbb{N}^{p+1}$,  denote $|\mathrm{k}|:=k_0+k_1+\cdots+k_{p}$. For $f \in C^{|\mathrm{k}|}(\Omega,\mathbb{O})$,  define the  \textit{partial derivative} $\partial_{\mathrm{k}}$ as
\begin{equation*}
\partial_{\mathrm{k}}f(\bx)=\partial_{\bx,\mathrm{k}}f(\bx)=\frac{\partial^{|\mathrm{k}|} }{\partial_{x_{0}}^{k_0}\partial_{x_{1}}^{k_1}\cdots\partial_{x_{p}}^{k_p} }f(\bx_p+\underline{\bx}_q), \quad \bx=\bx_p+\underline{\bx}_q.
\end{equation*}

By definition, we  have
\begin{proposition}\label{preserving-slice-monogenic}
Let $\Omega\subseteq \mathbb{O}$ be a  domain and $f\in \mathcal {GSM}(\Omega)$.  Then, for any $\by\in \mathbb{R}^{p+1}$ and  $\mathrm{k}\in\mathbb{N}^{p+1}$, we have
$$f(\cdot-\by) \in \mathcal {GSM}(\Omega+\{\by\}), \quad \partial_{\mathrm{k}}f \in \mathcal {GSM}(\Omega),$$
where $\Omega +\{\by\}=\{\bx\in \mathbb{O}: \bx-\by\in \Omega\}$.
\end{proposition}
\begin{proof}
Let $f\in \mathcal {GSM}(\Omega)$, $\by\in \mathbb{R}^{p+1}$ and  $\mathrm{k}\in\mathbb{N}^{p+1}$.
By hypothesis
$$D_{\underline{\omega}}f_{\underline{\omega}}(\bx)=(D_{\bx_p}+\underline{\omega}\partial_{r}) f_{\underline{\omega}}(\bx_p+r\underline{\omega})=0,$$
for all $\underline{\omega}\in \mathbb{S}$, making the change of variable $\widehat{\bx_p}=\bx_p-\by\in \mathbb{R}^{p+1}$, and observing that $D_{\widehat{\bx_p}}=D_{\bx_p}$, we have that
\[
\begin{split}(D_{\bx_p}+\underline{\omega}\partial_{r})  f_{\underline{\omega}}(\widehat{\bx_p}+r\underline{\omega})
&=(D_{\widehat{\bx_p}}+\underline{\omega}\partial_{r})  f_{\underline{\omega}}(\widehat{\bx_p}+r\underline{\omega})=0,  \end{split}
\]
and so we obtain the first assertion. Now, using the above equalities and Remark \ref{C1-C00}, we deduce that for all $\underline{\omega}\in \mathbb{S}$
$$(D_{\bx_p}+\underline{\omega}\partial_{r})  ( \partial_{\mathrm{k}}f_{\underline{\omega}})(\bx_p+r\underline{\omega})=\partial_{\mathrm{k}}((D_{\bx_p}+\underline{\omega}\partial_{r})  f_{\underline{\omega}}(\bx_p+r\underline{\omega}))=0, $$
i.e.
$\partial_{\mathrm{k}}f \in \mathcal {GSM}(\Omega)$. The proof is complete.
\end{proof}

\begin{proposition}\label{GSM-Harm}
A function $f\in \mathcal {GSM}(\Omega)$  is necessarily harmonic slice-by-slice on $\Omega$, that is  $f_{\underline{\omega}}$  is harmonic in $\Omega_{\underline{\omega}} $ for all $\underline{\omega}\in \mathbb S$.
\end{proposition}
\begin{proof}  Recalling   Remark \ref{C1-C00}, we have
  $f_{\underline{\omega} }\in C^{2}(\Omega_{\underline{\omega}}, \mathbb{O})$,   then by Proposition \ref{artin}, we deduce that
\begin{eqnarray*}
\Delta_{\bx'}   f_{\underline{\omega}} = (\overline{D}_{\underline{\omega}} D_{\underline{\omega}}) f_{\underline{\omega}}=\overline{D}_{\underline{\omega}} (D _{\underline{\omega}} f_{\underline{\omega}})=0,
 \end{eqnarray*}
where $\Delta_{\bx'}$ is Laplace operator in $\mathbb{R}^{p+2}$, and the assertion follows.
\end{proof}
The fact that all functions in  $\mathcal{GSM}(\Omega)$ are  harmonic slice-by-slice implies that they are real  analytic slice-by-slice, and also gives the following:
\begin{proposition}\label{Identity-theorem-monogenic-harmonic}
Let $\underline{\eta} \in \mathbb{S},\Omega\subseteq \mathbb{O}$ be a  domain and $f\in C^{1}(\Omega_{\underline{\eta}}, \mathbb{O})$ be a   function satisfying $D_{\underline{\eta}}f_{\underline{\eta}}=0$ in  $\Omega_{\underline{\eta}}$.
If $f_{\underline{\eta}}$ equals to zero in a ball  in $\Omega_{\underline{\eta}}$, then $f_{\underline{\eta}}\equiv 0$ in  $\Omega_{\underline{\eta}}$.
\end{proposition}
\begin{proof}
Given $\underline{\eta} \in \mathbb{S}$,  let us write   $f(\bx_p+ \underline{\eta} r)=\sum_{i=0}^7 f_i(x_0,x_1,\ldots, x_p,r) e_i$ in $\Omega_{\underline{\eta}}$, where the functions $f_i$ are real-valued. Since $f_{\underline{\eta}}$ is harmonic on $\Omega_{\underline{\eta}}$, it is also real analytic and thus all its real components are real analytic. Being $f_{\underline{\eta}}$ equal to zero in a ball  in $\Omega_{\underline{\eta}}$ also the real components $f_i$ vanish on that ball. By the identity principle for real analytic functions, see e.g. \cite{krantzparks}, we deduce that $f_i$ are zero on $\Omega_{\underline{\eta}}$ and so $f_{\underline{\eta}}=0$ in  $\Omega_{\underline{\eta}}$, as stated.
\end{proof}
Furthermore, following the strategy in \cite[Theorem 9.27]{Gurlebeck}, we can say more:
\begin{theorem}\label{Identity-theorem-monogenic}
Let $\underline{\eta} \in \mathbb{S}$, $\Omega\subseteq \mathbb{O}$ be a  domain and $f\in C^{1}(\Omega_{\underline{\eta}}, \mathbb{O})$  be a   function  satisfying $D_{\underline{\eta}}f_{\underline{\eta}}=0$ in  $\Omega_{\underline{\eta}}$.
If the set of zeros of $f_{ \underline{\eta}}$ contains  a   $(p+1)$-dimensional smooth manifold $M$ in $\Omega_{\underline{\eta}}$, then $f_{ \underline{\eta}}\equiv0$ in  $\Omega_{\underline{\eta}}$.
\end{theorem}
\begin{proof}
Let $\by$ denote an arbitrary point in $M\subseteq \Omega_{\underline{\eta}}$ and
$$x(\boldsymbol{t}_p) :=x(t_0,\ldots,t_p)=\sum _{i=0}^{p} x_{i}(\boldsymbol{t}_p) e_i+   x_{p+1}(\boldsymbol{t}_p)\underline{\eta} $$
  be a parametrization of $M$ in a neighborhood of $\by$ with
 $x(0)=\by$. In view of  $ f(x(\boldsymbol{t}_p))=0$ for all $\boldsymbol{t}_p$ and of the fact that the functions $x_i(\boldsymbol{t}_p)$ are real-valued, we have
  \begin{equation}\label{xjti}
  \sum _{i=0}^{p+1} \frac{\partial x_{i}}{\partial t_{j}} \frac{\partial f_{\underline{\eta}}}{\partial x_{i}}(\by)=0, \quad j=0,1,\ldots,p.
  \end{equation}
Since $M$ is $(p+1)$-dimensional,  we can  assume that  rank$(\frac{\partial x_{i}}{\partial t_{j}})=p+1.$ Hence, we can suppose that, without loss of generality, there exist some real numbers $a_0,\ldots, a_p$ such that
$$\frac{\partial x_{p+1}} {\partial t_j}= \sum_{i=0}^p a_i\frac{\partial x_{i}} {\partial t_j},$$
and substituting in \eqref{xjti}, we deduce that
$$
\sum _{i=0}^{p} \frac{\partial x_{i}}{\partial t_{j}} \frac{\partial f_{\underline{\eta}}}{\partial x_{i}}(\by)+ \sum_{i=0}^p a_i\frac{\partial x_{i}} {\partial t_j} \frac{\partial f_{\underline{\eta}}}{\partial x_{p+1}}(\by)  =0, \quad j=0,1,\ldots,p,
$$
so that
$$
\sum _{i=0}^{p} \frac{\partial x_{i}}{\partial t_{j}} \left(\frac{\partial f_{\underline{\eta}}}{\partial x_{i}}(\by)+  a_i \frac{\partial f_{\underline{\eta}}}{\partial x_{p+1}}(\by)\right)  =0, \quad j=0,1,\ldots,p.
$$
By setting $r_i=-a_i$, $i=0,\ldots,p$, we have $$ \frac{\partial f_{\underline{\eta}}}{\partial x_{i}}(\by)=r_i \frac{\partial f_{\underline{\eta}}}{\partial x_{p+1}}(\by).   $$
Hence
 $$D_{\underline{\eta}}f_{\underline{\eta}}(\by)= (\sum _{i=0}^{p} r_i e_i + \underline{\eta})\frac{\partial f_{\underline{\eta}}}{\partial x_{p+1}}(\by)=0.$$
Since $\mathbb{O}$ is a division algebra and $\sum_{i=0}^{p} r_i e_i + \underline{\eta} \neq 0$, we have  $\frac{\partial f_{\underline{\eta}}}{\partial x_{p+1}}(\by)=0$, and so
 $$\frac{\partial f_{\underline{\eta}}}{\partial x_{i}}(\by)=0, \quad i=0,1,\ldots,p.$$
 In conclusion, it has been proved that all first-order derivatives of the function $f_{\underline{\eta}}$ vanish in $M$.
Considering   $\frac{\partial f_{\underline{\eta}}}{\partial x_{i}}$ $(i=0,1,\ldots,p+1)$ and  iterating the process above, we deduce that all derivatives of the function $f_{\underline{\eta}}$ vanish in $M$. Consequently,   the coefficients of the Taylor series for the real analytic function $f_{\underline{\eta}}$ at some point $\by$ vanish,  so that $f_{\underline{\eta}}$   equals to zero in a suitable ball  in $\Omega_{\underline{\eta}}$. By Proposition \ref{Identity-theorem-monogenic-harmonic}, $f_{\underline{\eta}}=0$ in  $\Omega_{\underline{\eta}}$, as desired.
 \end{proof}

From  Example \ref{example-1},   the class of generalized partial-slice monogenic functions  turns out to be so large that, for $p\not=6$, may even contain discontinuous functions. Therefore, in this paper we consider functions defined some special domains described below.
\begin{definition} \label{slice-domain}
 Let $\Omega$ be a domain in $\mathbb{O}$.

1.   $\Omega$ is called  slice domain if $\Omega\cap\mathbb R^{p+1}\neq\emptyset$  and $\Omega_{\underline{\omega}}$ is a domain in $\mathbb{R}^{p+2}$ for every  $\underline{\omega}\in \mathbb{S}$.

2.   $\Omega$   is called  partially  symmetric with respect to $\mathbb R^{p+1}$ (p-symmetric for short) if, for   $\bx_{p}\in\R^{p+1}, r \in \mathbb R^{+},$ and $ \underline{\omega}  \in \mathbb S$,
$$\bx=\bx_p+r\underline{\omega} \in \Omega\Longrightarrow [\bx]:=\bx_p+r \mathbb S=\{\bx_p+r \underline{\omega}, \ \  \underline{\omega}\in \mathbb S\} \subseteq \Omega. $$
 \end{definition}

Denote by $\mathcal{Z}_{f}(\Omega)$  the zero set of the function $f:\Omega\subseteq\mathbb{O}\rightarrow \mathbb{O}$.   Using  Theorem \ref{Identity-theorem-monogenic},  we can prove an identity theorem for generalized partial-slice monogenic functions over slice domains.
\begin{theorem} \label{Identity-lemma}
Let $\Omega\subseteq \mathbb{O} $ be a  slice domain and $f: \Omega\rightarrow \mathbb{O}$ be a  generalized partial-slice monogenic function.
If there is an imaginary $\underline{\eta}  \in \mathbb S $ such that $\mathcal{Z}_{f}(\Omega) \cap \Omega_{\underline{\eta}}$ is a   $(p+1)$-dimensional smooth manifold, then $f\equiv0$ in  $\Omega$.
\end{theorem}
\begin{proof}
Let $f$ be a generalized partial-slice monogenic function on the slice domain $\Omega$.
Under the hypotheses that the zero set of $f$ in the domain $\Omega_{\underline{\eta}}$ is a   $(p+1)$-dimensional smooth manifold,   Theorem \ref{Identity-theorem-monogenic}  gives that $f_{\underline{\eta}}\equiv0$  in $\Omega_{\underline{\eta}}$. In particular, we get  that $f_{\underline{\eta}}$ vanishes on $\Omega_{\underline{\eta}} \cap \mathbb{R}^{p+1}$ which is nonempty since $\Omega$ is a slice domain; thus $f$ vanishes on $\Omega\cap\mathbb{R}^{p+1}\not=\emptyset$, and so for any $\underline{\omega}\in\mathbb{S}$ we get that $f_{\underline{\omega}}$  vanishes on the domain $\Omega_{\underline{\omega}}\cap\mathbb{R}^{p+1}$. Hence,  Theorem \ref{Identity-theorem-monogenic}  shows again that
 $f_{\underline{\omega}}\equiv0$ on $\Omega_{\underline{\omega}}$ for any $\underline{\omega}  \in \mathbb S$, i.e.,    $f\equiv0$ in  $\Omega$. The proof is complete.
\end{proof}

Theorem \ref{Identity-lemma} can be reformulated in the following statement.
\begin{theorem}  {\bf(Identity theorem)}\label{Identity-theorem}
Let $\Omega\subseteq \mathbb{O}$ be a  slice domain and $f,g:\Omega\rightarrow \mathbb{O}$ be   generalized partial-slice monogenic functions.
If there is an imaginary $\underline{\omega}  \in \mathbb S $ such that $f=g$ on a   $(p+1)$-dimensional smooth manifold in $\Omega_{\underline{\omega}}$, then $f\equiv g$ in  $\Omega$.
\end{theorem}

The identity theorem for generalized partial-slice monogenic  functions allows to establish a representation formula. To this end, we need a simple, yet very useful, lemma.
\begin{lemma}\label{a-w-b}
Let $a\in\mathbb{ R}^{p+1}, \underline{\omega}\in \mathbb{S}$. Then it holds that for any $b\in \mathbb{O}$
$$a(\underline{\omega}b)=\underline{\omega}(\overline{a}b).$$
\end{lemma}
\begin{proof}
Note that, for any $a,b,\underline{\omega}\in \mathbb{O}$,
$$[a,  \underline{\omega}, b]+[\overline{a},\underline{\omega}, b]= [a+\overline{a},  \underline{\omega},b]=0,$$
which gives
$$ [a,  \underline{\omega}, b]=-[\overline{a},\underline{\omega}, b]=[ \underline{\omega},\overline{a}, b].$$
Combining this with the fact
$$ a \underline{\omega} =\underline{\omega} \overline{a}, \quad a\in\mathbb{ R}^{p+1}, \underline{\omega}\in \mathbb{S},$$
we get the claim.
\end{proof}
\begin{theorem}  {\bf(Representation Formula)}  \label{Representation-Formula-SM}
Let $\Omega\subseteq \mathbb{O}$ be a p-symmetric slice domain and $f:\Omega\rightarrow \mathbb{O}$ be a  generalized partial-slice monogenic function.  Then, for any $\underline{\omega}\in \mathbb{S}$ and for $\bx_p+r\underline{\omega} \in \Omega$,
\begin{equation}\label{Representation-Formula-eq}
f(\bx_p+r \underline{\omega})=\frac{1}{2} (f(\bx_p+r\underline{\eta} )+f(\bx_p-r\underline{\eta}) )+
\frac{ 1}{2}\underline{\omega}(\underline{\eta} (  f(\bx_p-r\underline{\eta} )-f(\bx_p+r\underline{\eta}))),
\end{equation}
for any $\underline{\eta}\in \mathbb{S}$.

Moreover, the following two functions do not depend on $\underline{\eta}$:
$$F_1(\bx')=\frac{1}{2} (f(\bx_p+r\underline{\eta} )+f(\bx_p-r\underline{\eta} ) ),$$
$$F_2(\bx')=\frac{ 1}{2}\underline{\eta}(  f(\bx_p-r\underline{\eta} )-f(\bx_p+r\underline{\eta})).$$
\end{theorem}

\begin{proof} Consider a fixed $\underline{\eta}\in \mathbb{S}$ and the function defined by
\begin{eqnarray*}
 h(\bx)
=\frac{1}{2} (f(\bx_p+r\underline{\eta})+f(\bx_p-r\underline{\eta}) )+
\frac{ 1}{2} \underline{\omega}(\underline{\eta}(  f(\bx_p-r\underline{\eta})-f(\bx_p+r\underline{\eta})))
\end{eqnarray*}
for  $\bx=\bx_p+r\underline{\omega} \in \Omega$ with  $\bx_{p}\in\R^{p+1}, r \geq0,$ and  $\underline{\omega} \in \mathbb{S}$.

It is immediate that
 $f\equiv h$  in the domain  $\Omega\cap \mathbb{R}^{p+1}$. If we show that $h\in\mathcal {GSM}(\Omega)$, the result will follow from Theorem \ref{Identity-theorem}.
Applying Lemma \ref{a-w-b}, we have
$$D_{\bx_p} (\underline{\omega} (\underline{\eta}  f(\bx_p+r\underline{\eta})))=
\underline{\omega} ( \overline{D}_{\bx_p} (\underline{\eta}  f(\bx_p+r\underline{\eta})) )
 =\underline{\omega} (\underline{\eta}  (D_{\bx_p}   f(\bx_p+r\underline{\eta})) ),
 $$
and similarly,
$$D_{\bx_p} (\underline{\omega} (\underline{\eta}  f(\bx_p-r\underline{\eta})))
 =\underline{\omega} (\underline{\eta}  (D_{\bx_p}   f(\bx_p-r\underline{\eta})) ).
 $$
By Proposition \ref{artin}, $$
  (\underline{\omega}\partial_{r}) (\underline{\omega}(\underline{\eta}f(\bx_p+r\underline{\eta})))=
  (\underline{\omega}\partial_{r} \underline{\omega})(\underline{\eta}f(\bx_p+r\underline{\eta})))=
  -\underline{\eta} \partial_{r}(f(\bx_p+r\underline{\eta})),$$
and similarly,
 $$
  (\underline{\omega}\partial_{r}) (\underline{\omega}(\underline{\eta}f(\bx_p-r\underline{\eta})))=
  -\underline{\eta} \partial_{r}(f(\bx_p-r\underline{\eta})).$$
Hence, we have
 \begin{eqnarray*}
2  (\underline{\omega}\partial_{r}) h(\bx_p+r\underline{\omega})
=(\underline{\omega}\partial_{r}) (f(\bx_p+r\underline{\eta})+ f(\bx_p-r\underline{\eta}))-\underline{\eta}  \partial_r ( f(\bx_p-r\underline{\eta})-    f(\bx_p+r\underline{\eta})),
\end{eqnarray*}
and, in view of that  $f\in\mathcal {GSM}(\Omega)$ and Proposition \ref{artin},
 \begin{eqnarray*}
2   D_{\bx_p}  h(\bx_p+r\underline{\omega})&=&D_{\bx_p}(f(\bx_p+r\underline{\eta})+ f(\bx_p-r\underline{\eta}))\\
& &+
 \underline{\omega} (\underline{\eta}  (D_{\bx_p}   f(\bx_p-r\underline{\eta})-D_{\bx_p}   f(\bx_p+r\underline{\eta})) )
 \\
&= &D_{\bx_p}(f(\bx_p+r\underline{\eta})+ f(\bx_p-r\underline{\eta}))
\\
&&+  \underline{\omega} (\underline{\eta}  (  (\underline{\eta} \partial_r) (f(\bx_p-r\underline{\eta}))+ (\underline{\eta} \partial_r)f(\bx_p+r\underline{\eta})) )
\\
&=&D_{\bx_p}(f(\bx_p+r\underline{\eta})+ f(\bx_p-r\underline{\eta}))-\underline{\omega}       \partial_r ( f(\bx_p-r\underline{\eta})+    f(\bx_p+r\underline{\eta})),
\end{eqnarray*}
which gives that
 \begin{eqnarray*}
2  (D_{\bx_p}+\underline{\omega}\partial_{r}) h(\bx_p+r\underline{\omega})
=(D_{\bx_p}+\underline{\eta}\partial_{r})  (f(\bx_p+r\underline{\eta})+(D_{\bx_p}-\underline{\eta}\partial_{r})  (f(\bx_p-r\underline{\eta}),
\end{eqnarray*}
Now we immediately deduce that  $h\in\mathcal {GSM}(\Omega)$ if $f\in\mathcal {GSM}(\Omega)$ and the formula \eqref{Representation-Formula-eq} follows.

From   (\ref{Representation-Formula-eq}), we have
\begin{equation}\label{Representation-Formula-3}
f(\bx_p-r \underline{\omega})=\frac{1}{2} (f(\bx_p-r\underline{\eta} )+f(\bx_p+r\underline{\eta} ) )+
\frac{ 1}{2} \underline{\omega}(\underline{\eta}(  f(\bx_p+r\underline{\eta})-f(\bx_p-r\underline{\eta}))).
\end{equation}
Combining  (\ref{Representation-Formula-eq}) with (\ref{Representation-Formula-3}),  we obtain
\begin{equation}\label{F1}
f(\bx_p+r \underline{\omega} )+f(\bx_p-r \underline{\omega}  )=f(\bx_p+r\underline{\eta})+f(\bx_p-r\underline{\eta}),  \end{equation}
and
\begin{equation}\label{F2}
f(\bx_p+r \underline{\omega} )-f(\bx_p-r \underline{\omega}  )=\underline{\omega}(\underline{\eta}(  f(\bx_p-r\underline{\eta})-f(\bx_p+r\underline{\eta}))). \end{equation}
Here (\ref{F1}) means that $F_1$ does not depend on $\underline{\eta}$, and (\ref{F2}) gives by Proposition \ref{artin}
\begin{eqnarray*}
\underline{\omega}(f(\bx_p+r \underline{\omega} )-f(\bx_p-r \underline{\omega}  ))
&=&\underline{\omega}(\underline{\omega}(\underline{\eta}(  f(\bx_p-r\underline{\eta})-f(\bx_p+r\underline{\eta}))))
 \\
 &=&(\underline{\omega} \ \underline{\omega})(\underline{\eta}(  f(\bx_p-r\underline{\eta})-f(\bx_p+r\underline{\eta}))))
 \\
 &=&   \underline{\eta}(f(\bx_p+r\underline{\eta})-  f(\bx_p-r\underline{\eta}) ),
 \end{eqnarray*}
 which means that $F_2$ does not depend on $\underline{\eta}$. The proof is complete.
\end{proof}

As a corollary of the Representation Formula, we can present  an extension theorem that allows to construct a generalized partial-slice monogenic function starting from a function $f_{\underline{\eta}}$ that is defined in $\mathbb{R}^{p+1}+ \underline{\eta}\mathbb{R}$ for some $\underline{\eta}\in\mathbb S$ and is in the kernel of $(D_{\bx_p}+ \underline{\eta}\partial_r)$.

\begin{theorem}[Extension theorem]\label{extthm}
 Let $\Omega\subseteq \mathbb{O}$ be a p-symmetric slice domain.
Let $f_{\underline{\eta}}\in C^{1}(\Omega_{\underline{\eta}}, \mathbb{O})$ satisfying
$$(D_{\bx_p}+ \underline{\eta}\partial_r)f_{\underline{\eta}}(\bx_p+r\underline{\eta})=0, \qquad
\bx_p+r \underline{\eta}\in\Omega_{\underline{\eta}},
$$
for a given $\underline{\eta}\in\mathbb S$. Then, for any $\bx_p+\underline{\bx}_q=\bx_p+r\underline{\omega} \in\Omega$, the function defined by
\begin{equation*}
{\rm ext}(f_{\underline{\eta}})(\bx_p+ r\underline{\omega}):
=\frac{1}{2} (f(\bx_p+r\underline{\eta} )+f(\bx_p-r\underline{\eta}) )+
\frac{ 1}{2} \underline{\omega}(\underline{\eta} (  f(\bx_p-r\underline{\eta} )-f(\bx_p+r\underline{\eta})))
\end{equation*}
is the unique generalized partial-slice monogenic extension of $f_{\underline{\eta}}$ to the whole $\Omega$.
\end{theorem}
\begin{proof}
The fact that $f(\bx_p+r \underline{\omega}):={\rm ext}(f_{\underline{\eta}})(\bx_p+ r \underline{\omega})$ is generalized partial-slice monogenic follows from the computations in the proof of Theorem \ref{Representation-Formula-SM}. Since $f(\bx_p+r \underline{\eta})=f_{\underline{\eta}}(\bx_p+r\underline{\eta})$ the identity   theorem in Theorem \ref{Identity-theorem} implies that the extension is unique.
\end{proof}

\section{Cauchy-Pompeiu integral formula}
\label{Sect4}
To formulate Cauchy-Pompeiu integral formula in this non-associative case, we need   some technical lemmas.
\begin{lemma}\label{E-lemma}
Let $\underline{\eta}\in \mathbb{S}$ and   $\Omega\subseteq \mathbb{O}$ be a   domain. Consider  the function
\begin{equation}\label{slice-form-E-1-slice}
\phi(\bx_{p}+r\underline{\eta})=\Phi(\bx')+\underline{\eta} \Psi (\bx') \in C^1(\Omega_{\underline{\eta}}, \mathbb{O}),
\end{equation}
and assume that
  $\Phi(\bx') =\sum_{i=0}^p \Phi_i(\bx') e_i\in \mathbb{R}^{p+1}$ and $\Psi(\bx') \in \mathbb{R}$ satisfy
 \begin{equation}\label{i-j-E-1}
\partial_{x_i} \Phi_j= \partial_{x_j}\Phi_i, \quad 1\leq i,j\leq p,
\end{equation}
and
\begin{equation}\label{i-j-r}
 \partial_{r} \Phi_i= \partial_{x_i}\Psi, \quad 1\leq i\leq p.
\end{equation}
Then for all $a\in \mathbb{O}$
$$ D_{\underline{\eta}} (\phi_{\underline{\eta}} a)=(D_{\underline{\eta}}\phi_{\underline{\eta}} )a,$$
where $ D_{\underline{\eta}}= D_{\bx_p}+\underline{\eta}\partial_r$.
\begin{proof}
Recall that for all $a, b\in \mathbb{O}$,
$$[e_0, b, a]=0, $$
thus all the terms containing $e_0$ can be omitted in the calculations below.
We have that
\begin{eqnarray*}
[D_{\underline{\eta}}, \phi_{\underline{\eta}}, a]
&=&\sum_{i=1}^{p}[e_i, \partial_{x_i} \phi_{\underline{\eta}}, a] + [ \underline{\eta},   \partial_{r}\phi_{\underline{\eta}}, a]
 \\
 &=& \sum_{i=1}^{p} \Big( \sum_{j=1}^{p} [e_i,e_j, a] \partial_{x_i} \Phi_j + \sum_{j=p+1}^{7} [e_i,e_j, a] \partial_{x_i} ( \frac{x_j}{r}\Psi)\Big)+[ \underline{\eta},   \partial_{r}\Phi, a]
 \\
 &=& \sum_{i=1}^{p}\sum_{j=p+1}^{7} [e_i,e_j, a]\partial_{x_i}( \frac{x_j}{r}\Psi) +
 \sum_{j=p+1}^{7}\sum_{i=1}^{p}[e_j, e_i, a] \frac{x_j}{r} \partial_{r}\Phi_i
 \\
  &=& \sum_{i=1}^{p}\sum_{j=p+1}^{7} [e_i,e_j, a]\frac{x_j}{r}\partial_{x_i}\Psi +
 \sum_{j=p+1}^{7}\sum_{i=1}^{p}[e_j, e_i, a] \frac{x_j}{r} \partial_{r}\Phi_i
 \\
 &=& \sum_{i=1}^{p}\sum_{j=p+1}^{7} [e_i,e_j, a]\frac{x_j}{r} (\partial_{x_i}\Psi-\partial_{r}\Phi_i)
 \\
 &=&0,
\end{eqnarray*}
where the second, third,  and last equalities follows from  (\ref{slice-form-E-1-slice}), (\ref{i-j-E-1}) and (\ref{i-j-r}), respectively.
Hence, we infer that
$$[D_{\underline{\eta}}, \phi_{\underline{\eta}}, a]=0,$$
or, equivalently, $$ D_{\underline{\eta}} (\phi_{\underline{\eta}} a)=(D_{\underline{\eta}}\phi_{\underline{\eta}} )a.$$
The proof is complete.
\end{proof}
\end{lemma}

\begin{lemma}\label{Cauchy-formula-lemma-2-0}
  For any $a\in \mathbb{O}$, we have for all $\underline{\omega}\in \mathbb{S}$
$$ D_{\underline{\omega}} (E(\bx)a)=0, \quad  \bx=\bx_p+r\underline{\omega}\neq 0.$$
\end{lemma}
\begin{proof}
 Recall the Cauchy kernel in Example \ref{Cauchy-kernel-example}
 $$E(\bx)=\sum_{i=0}^{7}E_{i}(\bx)e_i, \quad E_{i}(\bx)=\frac{-1}{\sigma_{p+1}}\frac{x_i}{|\bx|^{p+2}}, \ i=1,\ldots,7, $$
It is immediate that the Cauchy kernel $E$ satisfies
 \begin{equation}\label{Eij}
\partial_j E_{i}= \frac{1}{\sigma_{p+1}}((p+2)\frac{x_ix_j}{|\bx|^{p+4}}-\frac{\delta_{ij}}{|\bx|^{p+2}})=\partial_i E_{j},\ i,j=1,\ldots,7,\end{equation}
and takes the form (\ref{slice-form-E-1-slice}) for any $\underline{\omega}\in \mathbb{S}$ with
$$ \Phi(\bx')=\frac{\overline{\bx_p}}{(|\bx_p|^{2}+r^{2})^{\frac{p+2}{2}}}, \   \Psi (\bx')=\frac{-r}{(|\bx_p|^{2}+r^{2})^{\frac{p+2}{2}}},$$
which   satisfy (\ref{i-j-r}).

Hence   Lemma \ref{E-lemma} gives, for all $\underline{\omega}\in \mathbb{S}$ and $a\in \mathbb{O}$,
$$ D_{\underline{\omega}} (E(\bx)a)=(D_{\underline{\omega}} E(\bx))a=0, \quad  \bx=\bx_p+r\underline{\omega}\neq 0,$$
which concludes  the proof. \end{proof}

\begin{lemma}\label{Cauchy-formula-lemma}
Let $\underline{\eta}\in \mathbb{S}$ and   $\Omega\subseteq \mathbb{O}$ be a   domain. If $\phi=\sum_{i=0}^{7}\phi_ie_i\in C^1(\Omega, \mathbb{O}) $  satisfies
\begin{equation}\label{i-j}
\partial_{x_i}\phi_j= \partial_{x_j}\phi_i, \quad 1\leq i\leq p,   1\leq   j\leq 7,
\end{equation}
and $\phi$ is of the form
\begin{equation}\label{slice-form-2}
\phi(\bx_{p}+r\underline{\eta})=\Phi(\bx')+\underline{\eta} \Psi (\bx'),
\end{equation}
 where   $\Phi(\bx')=\sum_{i=0}^p \phi(\bx')e_i \in  \mathbb{R}^{p+1}$  and $\Psi(\bx') \in \mathbb{R}$, then  for any $a\in \mathbb{O}$
$$\sum_{i=0}^{7}[e_i, D_{\underline{\eta}} \phi_i, a]=0.$$
\end{lemma}
\begin{proof}
Recall that for all $a, b\in \mathbb{O}$,
$[e_0, b, a]=0, $ so that we deduce the following chain of equalities
\begin{eqnarray*}
\sum_{i=0}^{7}[e_i, D_{\underline{\eta}} \phi_i, a]
&=&\sum_{i=1}^{7}[e_i, D_{\underline{\eta}} \phi_i, a]
\\
&=&\sum_{i=1}^{7}[e_i, D_{\bx_p} \phi_i, a] +\sum_{i=1}^{7}[e_i,  \underline{\eta}  \partial_{r}\phi_i, a]
 \\
 &=& \sum_{i=1}^{7} [e_i,\sum_{j=1}^{p}e_j  \partial_{x_j} \phi_i, a] +\sum_{i=1}^{7}[e_i\partial_{r}\phi_i,  \underline{\eta}, a]
 \\
 &=& \sum_{i=1}^{7} \sum_{j=1}^{p} [e_i,e_j, a]\partial_{x_j} \phi_i + [\partial_{r}\phi_{\eta},  \underline{\eta}, a].
\end{eqnarray*}
In view of (\ref{i-j}) and (\ref{slice-form-2}), we have
$$\sum_{1\leq i,j\leq p}   [e_i,e_j, a]\partial_{x_j} \phi_i =0, \   [\partial_{r}\phi_{\eta},  \underline{\eta}, a]=[\partial_{r}\Phi,  \underline{\eta}, a].$$
Hence,
\begin{eqnarray*}
\sum_{i=0}^{7}[e_i, D_{\underline{\eta}} \phi_i, a]
&=& \sum_{i=p+1}^{7} \sum_{j=1}^{p} [e_i,e_j, a]\partial_{x_j} \phi_i +\sum_{i=1}^{p}[e_i\partial_{r}\phi_i,  \underline{\eta}, a]
\\
&=&  \sum_{i=p+1}^{7} \sum_{j=1}^{p} [e_i,e_j, a]\partial_{x_j} \phi_i +\sum_{i=1}^{p}[e_i,  \underline{\eta}, a]\partial_{r}\phi_i
 \\
 &=& \sum_{i=p+1}^{7} \sum_{j=1}^{p} [e_i,e_j, a]\partial_{x_j} \phi_i +\sum_{i=1}^{p}\sum_{j=p+1}^{7} [e_i,  e_j, a] \frac{x_j}{r}\partial_{r}\phi_i
 \\
 &=& \sum_{i=p+1}^{7} \sum_{j=1}^{p} [e_i,e_j, a]\partial_{x_j} \phi_i +\sum_{i=1}^{p}\sum_{j=p+1}^{7} [e_i,  e_j, a] \partial_{x_j}\phi_i
 \\
 &=&\sum_{i=p+1}^{7} \sum_{j=1}^{p} [e_i,e_j, a]\partial_{x_j} \phi_i +\sum_{j=1}^{p}\sum_{i=p+1}^{7} [e_j,  e_i, a] \partial_{x_i}\phi_j.
\end{eqnarray*}
Finally,  recalling (\ref{i-j}), we get
$$\sum_{i=0}^{7}[e_i, D_{\underline{\eta}} \phi_i, a]=\sum_{i=p+1}^{7} \sum_{j=1}^{p} ([e_i,e_j, a]+[e_j,  e_i, a] ) \partial_{x_j} \phi_i=0,$$
which  completes the proof.
\end{proof}

\begin{lemma}\label{Cauchy-formula-lemma-E}
  For any $a\in \mathbb{O}$, we have
$$\sum_{i=0}^{7}[e_i, D_{\underline{\omega}} E_i(\bx), a]=0, \quad  \bx=\bx_p+r\underline{\omega}\neq 0.$$
\end{lemma}
\begin{proof}
Recalling  that the Cauchy kernel $E$  has the form (\ref{slice-form-E-1-slice}) and satisfies (\ref{Eij}), we  can
 conclude the proof using Lemma \ref{Cauchy-formula-lemma}.
\end{proof}

\begin{lemma}\label{Cauchy-formula-lemma-2}
Let $\underline{\eta}\in \mathbb{S}$ and   $U$ be a bounded domain in $\mathbb{O}$ with smooth boundary $\partial U_{\underline{\eta}}$. If  $ \phi, f\in C^1( \overline{U_{\underline{\eta}}},\mathbb{O})$, then
 $$\int_{\partial U_{\underline{\eta}}}   \phi    (\bn f)dS= \int_{ U_{\underline{\eta}}} \Big((\phi D_{\underline{\eta}}  )f + \phi (D_{\underline{\eta}}f) - \sum_{i=0}^{7} [e_i,  D_{\underline{\eta}} \phi_i, f]\Big)dV,
$$
where  $\bn=\sum_{i=0}^{p}n_i  e_i+n_{p+1} \underline{\eta}$ is the unit exterior normal to $\partial U_{\underline{\eta}}$,
 $dS$ and  $dV$ stand  for  the   classical   Lebesgue surface  element and volume element  in $\mathbb{R}^{p+2}$, respectively.
\end{lemma}
\begin{proof}
Let $\phi=\sum_{i=0}^{7}\phi_ie_i, f=\sum_{j=0}^{7}f_je_j \in C^1( \overline{U_{\underline{\eta}}},  \mathbb{O})$.
From the divergence theorem, it holds that for all real-valued functions $\phi_{i}, f_j\in C^1( \overline{U_{\underline{\eta}}})$
$$\int_{\partial U_{\underline{\eta}}}   \phi_{i}f_j   n_{k} dS
=\int_{U_{\underline{\eta}}}(   (\partial_{k}\phi_{i} )f_j + \phi_{i} (\partial_{k}f_j) ) dV,\quad k=0,1,\ldots,p,p+1. $$
By multiplying by $e_k, k=0,1,\ldots,p,$ and $\underline{\eta}$ on both sides of the above formula, respectively, and then taking summation, we get
$$\int_{\partial U_{\underline{\eta}}}   \phi_{i}f_j   \bn dS =
\int_{ U_{\underline{\eta}}} ( (D_{\underline{\eta}}\phi_{i} )f_j + \phi_{i} (D_{\underline{\eta}}f_j) )dV.$$
Multiplying by $e_j, j=0,1,\ldots,7,$ on the right side, we get
$$\int_{\partial U_{\underline{\eta}}}   \phi_{i}   \bn fdS =
\int_{ U_{\underline{\eta}}}  ((D_{\underline{\eta}}\phi_{i} )f + \phi_{i} (D_{\underline{\eta}}f) )dV,$$
which implies, by multiplying by $e_i, i=0,1,\ldots,7,$ on the left side and then taking  summation over $i$, the formula
$$\int_{\partial U_{\underline{\eta}}}   \phi    (\bn f)dS
=\int_{ U_{\underline{\eta}}} \Big( \sum_{i=0}^{7} e_i ((D_{\underline{\eta}}\phi_{i} )f) + \phi (D_{\underline{\eta}}f)\Big) dV,$$
which gives
$$\int_{\partial U_{\underline{\eta}}}   \phi    (\bn f)dS
=\int_{ U_{\underline{\eta}}}\Big( (\phi D_{\underline{\eta}}  )f  + \phi (D_{\underline{\eta}}f) -
\sum_{i=0}^{7} [e_i,  D_{\underline{\eta}} \phi_i, f]\Big)dV,$$
 as desired.
 \end{proof}

 Now we are in a position to  prove a slice version of the Cauchy-Pompeiu integral  formula.
\begin{theorem}[Cauchy-Pompeiu formula, I]\label{CauchyPompeiuslice}
Let $\underline{\eta}\in \mathbb{S}$ and   $U$ be a bounded domain in $\mathbb{O}$ with smooth boundary $\partial U_{\underline{\eta}}$. If $f\in C^1( \overline{U_{\underline{\eta}}}, \mathbb{O}) $, then
 $$f(\bx)=\int_{\partial U_{\underline{\eta}}}  E_{\by}(\bx) (\bn(\by)f(\by)) dS(\by)-
\int_{  U_{\underline{\eta}}} E_{\by}(\bx)  (D_{\underline{\eta}}f(\by)) dV(\by), \quad  \bx \in U_{\underline{\eta}},   $$
where $E_{\boldsymbol{y}}(\boldsymbol{x}):=E(\boldsymbol{y}-\boldsymbol{x})$, $\bn(\by)=\sum_{i=0}^{p}n_i (\by) e_i+n_{p+1}(\by) \underline{\eta}$ is the unit exterior normal to $\partial U_{\underline{\eta}}$ at $\by$,
 $dS$ and  $dV$ stand  for  the   classical   Lebesgue surface  element and volume element  in $\mathbb{R}^{p+2}$, respectively.
\end{theorem}
\begin{proof}
Given $\bx \in U_{\underline{\eta}}$,  denote $B(\bx,\epsilon)=\{\by \in \mathrm{H}_{\underline{\eta}}: |\by-\bx|<\epsilon\}$. Let $\phi(\by)=E_{\bx}(\by)=E(\bx-\by)=-E_{\by}(\bx)$. Then it holds that
$$\phi(\by) D_{\underline{\eta}}=0, \quad \forall\ \by\in \mathrm{H}_{\underline{\eta}}, \by\neq\bx.$$
For $\epsilon>0$ small enough, we have by  Lemmas \ref{Cauchy-formula-lemma-E} and  \ref{Cauchy-formula-lemma-2} for $f\in C^1( \overline{U_{\underline{\eta}}}, \mathbb{O}) $
 \begin{eqnarray*}
&& \int_{\partial U_{\underline{\eta}}} \phi    (\bn f)dS - \int_{\partial B(\bx,\epsilon)}   \phi    (\bn f)dS\\
&=& \int_{ U_{\underline{\eta}} \setminus B(\bx,\epsilon)}\Big( (\phi D_{\underline{\eta}}  )f + \phi (D_{\underline{\eta}}f) - \sum_{i=0}^{7} [e_i,  D_{\underline{\eta}} \phi_i, f]\Big)dV
\\
&=&\int_{ U_{\underline{\eta}} \setminus B(\bx,\epsilon)} \phi (D_{\underline{\eta}}f)   dV.
 \end{eqnarray*}
Recalling Lemma \ref{artin}, it follows that
 \begin{eqnarray*}
 \int_{\partial B(\bx,\epsilon) }   \phi    (\bn f)dS
&=&   \frac{1}{\sigma_{p+1}}\int_{\partial B(\bx,\epsilon) }  \frac{\overline{\by-\bx}}{|\by-\bx|^{p+2}}   (\frac{\by-\bx}{|\by-\bx|} f(\by))dS(\by)
\\
&=& \frac{1}{\sigma_{p+1}}\int_{\partial B(\bx,\epsilon) } \Big( \frac{\overline{\by-\bx}}{|\by-\bx|^{p+2}}   \frac{\by-\bx}{|\by-\bx|}\Big) f(\by)dS(\by)
\\
&=& \frac{1}{\epsilon^{p+1}\sigma_{p+1}}\int_{\partial B(\bx,\epsilon) }   f(\by)dS(\by)
\\
&\rightarrow &   f(\bx), \ \epsilon\rightarrow0.
 \end{eqnarray*}
Combining the two facts above, we get
 \begin{eqnarray*}
 \int_{\partial U_{\underline{\eta}}}   \phi    (\bn f)dS-f(\bx)
=\lim_{\epsilon\rightarrow0} \int_{ U_{\underline{\eta}} \setminus B(\bx,\epsilon)} \phi (D_{\underline{\eta}}f)   dV
=\int_{ U_{\underline{\eta}}  } \phi (D_{\underline{\eta}}f)   dV,
\end{eqnarray*}
i.e.
 $$f(\bx)=\int_{\partial U_{\underline{\eta}}}  E_{\by}(\bx) (\bn(\by)f(\by)) dS(\by)-
\int_{  U_{\underline{\eta}}} E_{\by}(\bx)  (D_{\underline{\eta}}f(\by)) dV(\by).$$
 The proof is complete.
\end{proof}

As a special case of Theorem  \ref{CauchyPompeiuslice}, we have
\begin{theorem}[Cauchy formula, I]\label{Cauchy-slice}
Let $\underline{\eta}\in \mathbb{S}$ and     $U$ be a bounded domain in $\mathbb{O}$ with smooth boundary $\partial U_{\underline{\eta}}$.  If $ f\in C^1( \overline{U_{\underline{\eta}}}, \mathbb{O}) $  satisfies  $D_{\underline{\eta}}f(\by)=0$ for all $\by \in U_{\underline{\eta}}$, then
 $$f(\bx)=\int_{\partial U_{\underline{\eta}}}  E_{\by}(\bx) (\bn(\by)f(\by)) dS(\by), \quad  \bx \in U_{\underline{\eta}},   $$
where  $\bn(\by)=\sum_{i=0}^{p}n_i (\by) e_i+n_{p+1}(\by) \underline{\eta}$ is the unit exterior normal to $\partial U_{\underline{\eta}}$ at $\by$,
 $dS$ and  $dV$ stand  for  the   classical   Lebesgue surface  element and volume element  in $\mathbb{R}^{p+2}$, respectively.
\end{theorem}

In view of Lemma \ref{Cauchy-formula-lemma-2-0}, we can present an inverse of the Cauchy formula.
\begin{theorem}\label{Cauchy-slice-inverse}
Let $\underline{\eta}\in \mathbb{S}$ and     $U$ be a bounded domain in $\mathbb{O}$ with smooth boundary $\partial U_{\underline{\eta}}$.  For
 $ g\in C( \partial U_{\underline{\eta}}, \mathbb{O}) $, define
 $$f(\bx)=\int_{\partial U_{\underline{\eta}}}  E_{\by}(\bx) (\bn(\by)g(\by)) dS(\by), \quad  \bx \in U_{\underline{\eta}},   $$
where  $\bn(\by)=\sum_{i=0}^{p}n_i (\by) e_i+n_{p+1}(\by) \underline{\eta}$ is the unit exterior normal to $\partial U_{\underline{\eta}}$ at $\by$,
 $dS$ and  $dV$ stand  for  the   classical   Lebesgue surface  element and volume element  in $\mathbb{R}^{p+2}$, respectively.
 Then  $D_{\underline{\eta}}f(\bx)=0$ for all $\bx \in U_{\underline{\eta}}$.
\end{theorem}

Theorem \ref{Cauchy-slice} allows  to obtain several consequences as it happens in the classical case of holomorphic functions, among which the  mean value theorem and the maximum modulus principle.
\begin{theorem}[Mean value theorem]\label{mean}
Let $\underline{\eta}\in \mathbb{S}$ and     $U$ be a domain in $\mathbb{O}$ with smooth boundary $\partial U_{\underline{\eta}}$.  If $f\in C^1(U_{\underline{\eta}}, \mathbb{O}) $ satisfies  $D_{\underline{\eta}}f(\by)=0$ for all $\by \in U_{\underline{\eta}}$, then
 $$f(\bx)=\frac{1}{\sigma_{p+1}  \epsilon^{p+1}} \int_{\partial B(\bx,\epsilon) }  f(\by) dS(\by), \quad \bx \in U_{\underline{\eta}},  $$
where $B(\bx,\epsilon)=\{\by \in \mathrm{H}_{\underline{\eta}}: |\by-\bx|<\epsilon\} \subset U_{\underline{\eta}}$.
\end{theorem}
Theorem \ref{mean} can be obtained easily from Theorem \ref{Cauchy-slice} and its proof,  or directly from the fact that all functions in $\mathcal{GSM}$ are  harmonic   slice-by-slice by Proposition \ref{GSM-Harm}.

\begin{theorem}{\bf (Maximum modulus principle)}
Let $\Omega\subseteq \mathbb{O}$ be a  slice domain and $f:\Omega\rightarrow \mathbb{O}$ be a  generalized partial-slice monogenic function.  If $|f|$ has a relative maximum at some point in $\Omega$, then $f$ is constant.
\end{theorem}
\begin{proof}
 Assume that $|f|$ has a relative maximum at   $\bx =\bx_p+r  \underline{\omega}\in \Omega$ for some $\underline{\omega}\in \mathbb{S}$ and let $ \rho>0$ be small enough such that  $B_{\underline{{\omega}} }\subset \Omega_{\underline{{\omega}}} $, where $B=B(\bx,\rho)=\{\by \in \mathbb{R}^{p+q+1}: |\by-\bx|<\rho \}$.     By Theorem \ref{mean},
it follows that
 $$f(\bx)=\frac{1}{\sigma_{p+1}\rho^{p+1}}\int_{\partial B_{\underline{\omega}}} f(\by) dS(\by)=\frac{1}{\sigma_{p+1}\rho^{p+1}}\int_{\partial (B(0,\rho)_{\underline{\omega}})} f(\bx+\by) dS(\by),   $$
which implies that
$$|f(\bx)|\leq \frac{1}{\sigma_{p+1}\rho^{p+1}}\int_{\partial (B(0,\rho)_{\underline{\omega}})} |f(\bx+\by)| dS(\by)\leq|f(\bx)|.   $$
The above inequality forces that  $|f_{\underline{\omega} }|$   is  constant in a small neighbourhood of $\bx$ in $\Omega_{\underline{\omega}}$. Now let us show $f$ is constant. To see this,   write
$f_{\underline{\omega}}=\sum_{i=0}^{7}f_{i}e_{i},  \quad f_{i}\in \mathbb{R}$.
Since $|f_{\underline{\omega}}|^{2}=\sum_{i=0}^{7}f_{i}^{2}$ is constant, the derivatives of $|f_{\underline{\omega}}|^{2}$ with respect to variable $x_{j},j=0,1,\ldots,p+1,$ are zero, namely
$$\sum_{i=0}^{7} f_i(\partial_{x_j}f_{i})=0.$$
A second differentiation with  respect to the variable $x_{j},j=0,1,\ldots,p+1,$  and taking summation give that
$$0=\sum_{i=0}^{7} \sum_{j=0}^{p+1}((\partial_{x_j}f_{i})^{2}+f_{i}\partial_{x_j}^{2} f_{i})
=\sum_{i=0}^{7} \sum_{j=0}^{p+1}(\partial_{x_j}f_{i})^{2}+\sum_{i=0}^{7}f_{i} \Delta_{\bx'} f_{A}=\sum_{i=0}^{7} \sum_{j=0}^{p+1}(\partial_{x_j}f_{i})^{2},$$
where $\Delta_{\bx'}$ is the Laplacian   in $\mathbb{R}^{p+2}$.

Consequently, all $f_{i}, i=0,1,\ldots,7$ are  constant in $\Omega_{\underline{\omega}}$, and so is  $f_{\underline{\omega}}$. Therefore,
 $f$ is constant in $\Omega$  by the identity theorem in Theorem \ref{Identity-theorem}.
\end{proof}

\section{Generalized partial-slice  functions}\label{Sec4}

An open set $D$ of $\mathbb{R}^{p+2}$  is called invariant under the reflection  of the $(p+2)$-th variable if
$$ \bx'=(\bx_p,r) \in D \Longrightarrow   \bx_\diamond':=(\bx_p,-r)  \in D.$$
The  \textit{p-symmetric completion} $ \Omega_{D}$ of  $D$ is defined by
$$\Omega_{D}=\bigcup_{\underline{\omega} \in \mathbb{S}} \, \big \{\bx_p+r\underline{\omega}\  : \ \exists \ \bx_p \in \mathbb{R}_{p}^{0}\oplus \mathbb{R}_{p}^{1},\ \exists \ r\geq 0,\  \mathrm{s.t.} \ (\bx_p,r)\in D \big\}.$$
\begin{definition}
A function $F=(F_1,F_2): D\longrightarrow  \mathbb{O}^{2}$ in an open set $D\subseteq  \mathbb{R}^{p+2}$, which is invariant under the reflection  of the $(p+2)$-th variable, is called  a \textit{stem function} if
the $\mathbb{O}$-valued components  $F_1, F_2$  satisfy
$$ F_1(\bx_{\diamond}')= F_1(\bx'), \qquad  F_2(\bx_{\diamond}')=-F_2(\bx'), \qquad  \bx'=(\bx_p,r) \in D.$$
Each stem function $F$ induces a (left)  generalized partial-slice
function $f=\mathcal I(F): \Omega_{D} \longrightarrow \mathbb{O }$ given by
 $$f(\bx)=\mathcal I(F)(\bx):=F_1(\bx')+\underline{\omega} F_2(\bx'), \qquad   \boldsymbol{x}=\boldsymbol{x}_p+r\underline{\omega}   \in \Omega_{D}.$$
\end{definition}

Denote the set of  all induced  generalized partial-slice functions  on $\Omega_{D}$  by
 $$ {\mathcal{GS}}(\Omega_{D}):=\Big\{f=\mathcal I(F):    \ F \ {\mbox {is a stem function on }} D  \Big\},$$
and further set
 $${\mathcal{GS}}^{j}(\Omega_{D}):=\Big\{f=\mathcal I(F):    \ F \ {\mbox {is a}}\  C^{j} \ {\mbox {stem function on }} D  \Big\}, \quad j=0,1.$$

Now  we can establish the following formula  for generalized partial-slice functions.
\begin{theorem}  {\bf(Representation Formula, II)}  \label{Representation-Formula-SR}
Let $f\in {\mathcal{GS}}(\Omega_{D})$.  Then it holds that, for every  $\bx=\bx_p+r\underline{\omega} \in \Omega_{D}$ with $\underline{\omega}\in \mathbb{S}$,
\begin{equation*}\label{Rf of slice}
f(\bx)=(\underline{\omega}-\underline{\omega}_{2})((\underline{\omega}_{1}-\underline{\omega}_{2})^{-1}f(\bx_p+r\underline{\omega}_{1}) ) -(\underline{\omega}-\underline{\omega}_{1})((\underline{\omega}_{1}-\underline{\omega}_{2})^{-1}f(\bx_p+r\underline{\omega}_{2})),
\end{equation*}
 for all $\underline{\omega}_{1}\neq\underline{\omega}_{2}\in \mathbb{S}$.
 In particular, $\underline{\omega}_{1}=-\underline{\omega}_{2}=\underline{\eta}\in \mathbb{S},$
\begin{eqnarray*}
 f(\bx)= \frac{1}{2} (f(\bx_p+r\underline{\eta} )+f(\bx_p-r\underline{\eta} ) )+
\frac{ 1}{2} \underline{\omega} (\underline{\eta}(  f(\bx_p-r\underline{\eta} )-f(\bx_p+r\underline{\eta}))).
\end{eqnarray*}
\end{theorem}
\begin{proof}
Let $\bx=\bx_p+r\underline{\omega} \in \Omega_{D}$ with $\underline{\omega}\in \mathbb{S}$.  By definition, it follows that, for all  $\underline{\omega}_{1},\underline{\omega}_{2}\in \mathbb{S}$,
 $$f(\bx_p+r\underline{\omega}_{1})=F_1(\bx')+\underline{\omega}_{1} F_2(\bx'),$$
and
$$f(\bx_p+r\underline{\omega}_{2})=F_1(\bx')+\underline{\omega}_{2} F_2(\bx').$$
Hence, for $\underline{\omega}_{1}\neq\underline{\omega}_{2},$
 $$ F_2(\bx')=(\underline{\omega}_{1}-\underline{\omega}_{2})^{-1}(f(\bx_p+r\underline{\omega}_{1})-f(\bx_p+r\underline{\omega}_{2})),$$
and then
  \begin{eqnarray*}
 F_1(\bx')
&=&f(\bx_p+r\underline{\omega}_{2})-\underline{\omega}_{2}F_2(\bx')
 \\
 &=&f(\bx_p+r\underline{\omega}_{2})-\underline{\omega}_{2}((\underline{\omega}_{1}-\underline{\omega}_{2})^{-1}
 (f(\bx_p+r\underline{\omega}_{1})-f(\bx_p+r\underline{\omega}_{2})))
 \\
 &=&\underline{\omega}_{1}((\underline{\omega}_{1}-\underline{\omega}_{2})^{-1}f(\bx_p+r\underline{\omega}_{2}))
 -\underline{\omega}_{2}((\underline{\omega}_{1}-\underline{\omega}_{2})^{-1}
 f(\bx_p+r\underline{\omega}_{1})),
 \end{eqnarray*}
where the third equality follows from Proposition \ref{artin},
Therefore
  \begin{eqnarray*}
f(\bx)
&=&F_1(\bx')+\underline{\omega}F_2(\bx')
 \\
 &=&\underline{\omega}_{1}((\underline{\omega}_{1}-\underline{\omega}_{2})^{-1}f(\bx_p+r\underline{\omega}_{2}))-
 \underline{\omega}_{2}((\underline{\omega}_{1}-\underline{\omega}_{2})^{-1}
 f(\bx_p+r\underline{\omega}_{1}))
 \\
 & &+\underline{\omega}((\underline{\omega}_{1}-\underline{\omega}_{2})^{-1}(f(\bx_p+r\underline{\omega}_{1})-
 f(\bx_p+r\underline{\omega}_{2})))
 \\
 &=&(\underline{\omega}-\underline{\omega}_{2})((\underline{\omega}_{1}-\underline{\omega}_{2})^{-1}f(\bx_p+r\underline{\omega}_{1}) ) -(\underline{\omega}-\underline{\omega}_{1})((\underline{\omega}_{1}-\underline{\omega}_{2})^{-1}f(\bx_p+r\underline{\omega}_{2})),
\end{eqnarray*}
which completes the proof.\end{proof}

\begin{remark}{\rm
We now highlight a feature which is typical of the non-associative framework.
Note that, for $a, \by \in \mathbb{O}$ and  $\underline{\omega}, \underline{\eta} \in \mathbb{S}$, the  following two  terms
$$  \underline{\omega} (\underline{\eta}( E_{\by}(\bx_p-r \underline{\eta}) a ))
- \underline{\omega} (\underline{\eta}( E_{\by}(\bx_p+r \underline{\eta}) a)$$
and $$ \big( \underline{\omega} (\underline{\eta}E_{\by}(\bx_p-r \underline{\eta}))\big) a
- \big(\underline{\omega} (\underline{\eta} E_{\by}(\bx_p+r \underline{\eta}) )\big) a $$
do not coincide generally.

 In view of this observation, we need to carefully define the extended kernel appearing in the Cauchy formula.
 }
\end{remark}
For $a\in \mathbb{O}$,  define  the operator of left multiplication $L_{a}: \mathbb{O}\rightarrow\mathbb{O}$   given by
$$L_{a}x=ax,\quad x\in \mathbb{O}.$$

\begin{definition} Given $\by \in \mathbb{R}^{p+q+1}$,  define the   operator $\mathcal{E}_{\by}(\bx): \mathbb{O}\rightarrow  \mathbb{O}$ as
\begin{equation*}\label{slice-Cauchy-kernel}
\mathcal{E}_{\by}(\bx)=\frac{1}{2}L_{( E_{\by}(\pi_{\by}(\bx))+E_{\by}( \pi_{\by}(\bx)_{\diamond}))}+
\frac{1}{2}  L_{\underline{\omega}} L_{\underline{\eta}}  L_{ (E_{\by}( \pi_{\by}(\bx)_{\diamond})-E_{\by}( \pi_{\by}(\bx)))},
\end{equation*}
where $\bx=\bx_p+r  \underline{\omega},$ $\by=\by_p+\widetilde{r}  \underline{\eta},
\pi_{\by}(\bx)=\bx_p+r \underline{\eta} $  and $\pi_{\by}(\bx)_{\diamond}=\bx_p-r\underline{\eta}$.
\end{definition}

Now we can prove a global   version of Cauchy-Pompeiu integral formula.
\begin{theorem}[Cauchy-Pompeiu  formula, II]\label{Cauchy-Pompeiu}
Let $f\in {\mathcal{GS}}^{1}(\Omega_{D})$ and set $\Omega=\Omega_{D}$. If  $U$ is a domain in  $\mathbb{O}$ such that  $U_{\underline{\eta} }\subset \Omega_{\underline{\eta}} $  is a bounded domain in $\mathbb{R}^{p+2}$ with    smooth boundary $\partial U_{\underline{\eta}}\subset\Omega_{\underline{\eta}}$ for some $\underline{\eta}\in \mathbb{S}$, then
$$f(\bx)=\int_{\partial U_{\underline{\eta}}} \mathcal{E}_{\by}(\bx) (\bn(\by)f(\by)) dS(\by)-
\int_{  U_{\underline{\eta}}} \mathcal{E}_{\by}(\bx)  (D_{\underline{\eta}}f(\by)) dV(\by), \quad  \bx \in U,   $$
where  $\bn(\by)=\sum_{i=0}^{p}n_i(\by) e_i+n_{p+1}(\by)\underline{\eta}$ is the unit exterior normal to $\partial U_{\underline{\eta}}$ at $\by$,
 $dS$ and  $dV$ stand  for  the   classical   Lebesgue surface  element and volume element  in $\mathbb{R}^{p+2}$, respectively.
\end{theorem}

\begin{proof}
Let $f\in {\mathcal{GS}}^{1}(\Omega_{D})$ and set $\Omega=\Omega_{D}$.
For a given    $\underline{\eta} \in \mathbb{S}$, we have by Theorem \ref{CauchyPompeiuslice} for
$\bx_p \pm r\underline{\eta}\in U_{\underline{\eta} }$
$$f(\bx_p \pm r\underline{\eta})=\int_{\partial U_{\underline{\eta}}} E_{\by}(\bx_p\pm r \underline{\eta}) (\bn(\by) f(\by)) dS(\by)-
\int_{U_{\underline{\eta}}} E_{\by}(\bx_p\pm r\underline{\eta})  (D_{\underline{\eta}}f(\by)) dV(\by),   $$
where $U$ is a domain in  $\mathbb{O}$ such that  $U_{\underline{\eta} }\subset \Omega_{\underline{\eta}} $  is a bounded domain in $\mathbb{R}^{p+2}$ with    smooth boundary $\partial U_{\underline{\eta}}\subset\Omega_{\underline{\eta}}$.

Hence, by using  the Representation Formula  in Theorem \ref{Representation-Formula-SR},  we  obtain for all $\underline{\omega} \in \mathbb{S}$ and  $\bx=\bx_p+r\underline{\omega} \in U$
 \[
   \begin{split}
 &\ 2 f(\bx) \\
 &=  f(\bx_p+r\underline{\eta} )+f(\bx_p-r\underline{\eta} ) +
  \underline{\omega} (\underline{\eta}(  f(\bx_p-r\underline{\eta} )-f(\bx_p+r\underline{\eta})))
   \\
   &= \int_{\partial U_{\underline{\eta}}} E_{\by}(\bx_p+r \underline{\eta}) (\bn(\by) f(\by)) dS(\by)-
\int_{U_{\underline{\eta}}} E_{\by}(\bx_p+r\underline{\eta})  (D_{\underline{\eta}}f(\by)) dV(\by)
\\
    &+\int_{\partial U_{\underline{\eta}}} E_{\by}(\bx_p-r \underline{\eta}) (\bn(\by) f(\by)) dS(\by)-
\int_{U_{\underline{\eta}}} E_{\by}(\bx_p-r\underline{\eta})  (D_{\underline{\eta}}f(\by)) dV(\by)\\
  &+\int_{\partial U_{\underline{\eta}}} \underline{\omega} (\underline{\eta}( E_{\by}(\bx_p-r \underline{\eta}) (\bn(\by) f(\by)) ))dS(\by)-
\int_{U_{\underline{\eta}}}  \underline{\omega} (\underline{\eta}( E_{\by}(\bx_p-r\underline{\eta})  (D_{\underline{\eta}}f(\by)) ))dV(\by)\\
  &-\int_{\partial U_{\underline{\eta}}} \underline{\omega} (\underline{\eta}( E_{\by}(\bx_p+r \underline{\eta}) (\bn(\by) f(\by)) ))dS(\by)
  +\int_{U_{\underline{\eta}}}  \underline{\omega} (\underline{\eta}( E_{\by}(\bx_p+r\underline{\eta})  (D_{\underline{\eta}}f(\by)) ))dV(\by)\\
& =2\int_{\partial U_{\underline{\eta}}} \mathcal{E}_{\by}(\bx) (\bn(\by)f(\by)) dS(\by)-
2\int_{  U_{\underline{\eta}}} \mathcal{E}_{\by}(\bx)  (D_{\underline{\eta}}f(\by)) dV(\by).
\end{split}
  \]
The proof is complete.
\end{proof}
Theorems \ref{Cauchy-Pompeiu}  and    \ref{Representation-Formula-SM} give  a general Cauchy integral formula for generalized partial-slice monogenic functions.
\begin{theorem}[Cauchy  formula, II]\label{Cauchy-formula-slice}
Let $\Omega\subseteq \mathbb{O}$ be a  p-symmetric slice domain and $f:\Omega\rightarrow \mathbb{O}$ be a  generalized partial-slice monogenic function. Given any $\underline{\eta} \in \mathbb{S}$,  let $U_{\underline{\eta} }\subset \Omega_{\underline{\eta}} $ be a   bounded domain in $\mathbb{R}^{p+2}$ with smooth boundary $\partial U_{\underline{\eta}}\subset\Omega_{\underline{\eta}}$. Then
$$f(\bx)=\int_{\partial U_{\underline{\eta}}} \mathcal{E}_{\by}(\bx) (\bn(\by)f(\by)) dS(\by), \quad  \bx \in U,   $$
where  the integral does not depend on the choice of $\underline{\eta}$, $\bn(\by)=\sum_{i=0}^{p}n_i(\by) e_i+n_{p+1}(\by)\underline{\eta}$ is the unit exterior normal to $\partial U_{\underline{\eta}}$ at $\by$ and $dS$  stands  for  the   classical   Lebesgue surface  element in $\mathbb{R}^{p+2}$.
\end{theorem}

\begin{definition}\label{definition-GSR}
Let $f \in {\mathcal{GS}}^{1}(\Omega_{D})$. The function $f$ is called generalized partial-slice regular of type $p$ if its stem function $F=(F_1,F_2)\in \mathbb{O}^{2}$ satisfies  the generalized Cauchy-Riemann equations
 \begin{eqnarray}\label{C-R}
 \left\{
\begin{array}{ll}
D_{\bx_p}  F_1- \partial_{r} F_2=0,
\\
 \overline{D}_{\bx_p}  F_2+ \partial_{r} F_1=0.
\end{array}
\right.
\end{eqnarray}
\end{definition}
Denote by $\mathcal {GSR}(\Omega_D)$ the set of all generalized partial-slice regular functions  on $\Omega_D$.
As before, the type $p$ will be omitted in the sequel.

Now we present a relationship between the set of functions  $\mathcal {GSM}$ and $\mathcal {GSR}$ defined in  p-symmetric domains.

\begin{theorem} \label{relation-GSR-GSM}

(i) For a p-symmetric domain $\Omega=\Omega_{D}$ with $\Omega  \cap \mathbb{R}^{p+1}= \emptyset$, it holds that $\mathcal {GSM}(\Omega) \supsetneqq \mathcal {GSR}(\Omega_{D})$.

(ii) For a p-symmetric domain $\Omega=\Omega_{D}$ with $\Omega  \cap \mathbb{R}^{p+1}\neq \emptyset$,  it holds  that $\mathcal {GSM}(\Omega) = \mathcal {GSR}(\Omega_{D})$.
\end{theorem}

 \begin{proof}
$(i)$  Let $f=\mathcal I(F)\in  \mathcal {GSR}(\Omega_{D})$ with   its stem  function $F\in C^{1}(D)$ satisfying  the generalized Cauchy-Riemann equations
 \begin{eqnarray*}
 \left\{
\begin{array}{ll}
D_{\bx_p}  F_1- \partial_{r} F_2=0,
\\
 \overline{D}_{\bx_p}  F_2+ \partial_{r} F_1=0.
\end{array}
\right.
\end{eqnarray*}
 First, we have from Lemma \ref{a-w-b}
  \begin{eqnarray}\label{D-w-F}
 D_{\bx_p}( \underline{\omega} F_2(\bx'))=\underline{\omega} (\overline{D}_{\bx_p}  F_2(\bx')),
\end{eqnarray}
and
  \begin{eqnarray}\label{D-ba-w-F}
   \overline{D}_{\bx_p}( \underline{\omega} F_2(\bx'))=\underline{\omega} (D_{\bx_p}  F_2(\bx')). \end{eqnarray}
 Then by Proposition \ref{artin} and (\ref{D-w-F})
 \begin{eqnarray*}
 D_{\underline{\omega}}f(\bx) &=& (D_{\bx_p}+\underline{\omega}\partial_{r})(F_1(\bx')+\underline{\omega} F_2(\bx'))
 \\
 &=& D_{\bx_p}  F_1(\bx')+\underline{\omega} (\underline{\omega} \partial_{r} F_2(\bx'))+ D_{\bx_p}( \underline{\omega} F_2(\bx')) +\underline{\omega} \partial_{r} F_1(\bx')
 \\
 &=& D_{\bx_p}  F_1(\bx')- \partial_{r} F_2(\bx')+ \underline{\omega}(\overline{D}_{\bx_p} F_2(\bx')  +  \partial_{r} F_1(\bx'))
 \\
 &=&0,
\end{eqnarray*}
which means that $f\in\mathcal {GSM}(\Omega)$.

To see that the inclusion is strict, consider the function
$$f(\bx) = \underline{\omega}, \quad \bx \in\Omega=\mathbb{O}\setminus \mathbb{R}^{p+1},$$
where $\bx=\bx_p+r\underline{\omega}$ with   $\bx_p\in \mathbb{R}^{p+1}, r>0, \underline{\omega}\in \mathbb{S}$.

It is immediate that $f\in  \mathcal {GSM}(\Omega)$ but $f\notin\mathcal {GSR}(\Omega)$.

$(ii)$  From the proof of   (i), we have the inclusion  $\mathcal {GSM}(\Omega) \supseteq \mathcal {GSR}(\Omega_{D})$ for $\Omega=\Omega_{D}$,  It remains to show   $\mathcal {GSM}(\Omega) \subseteq \mathcal {GSR}(\Omega_{D})$. Let $f\in  \mathcal {GSM}(\Omega)$.  Note that when the p-symmetric domain $\Omega_{D}$  is  slice, the Representation Formula in Theorem \ref{Representation-Formula-SM} holds
$$f(\bx)=f(\bx_p+r \underline{\omega})=  F_1(\bx_p,r)+\underline{\omega}F_2(\bx_p,r),$$
where $F_1, F_2$ are defined as is  Theorem \ref{Representation-Formula-SM}.
In fact,  $(F_1, F_2)$ is a stem function and satisfies  the generalized Cauchy-Riemann equations:
 \begin{eqnarray*}
D_{\bx_p}  F_1 (\bx_p,r)&=& \frac{1}{2} (D_{\bx_p} f(\bx_p+r\underline{\omega} )+D_{\bx_p}f(\bx_p-r\underline{\omega} )  )
\\
 &=&  \frac{1}{2} ((- \underline{\omega}\partial_{r})  f(\bx_p+r\underline{\omega}) +(\underline{\omega}\partial_{r})  f(\bx_p-r\underline{\omega})  )
 \\
&=& \partial_{r} F_2 (\bx_p,r),
\end{eqnarray*}
and by (\ref{D-ba-w-F}) and Proposition \ref{artin},
 \begin{eqnarray*}
\overline{D}_{\bx_p}  F_2 (\bx_p,r)&=&\frac{1}{2} \underline{\omega} (D_{\bx_p}( f(\bx_p-r\underline{\omega} )- f(\bx_p+r\underline{\omega} )  ))
\\
 &=&\frac{1}{2} \underline{\omega} ((\underline{\omega}\partial_{r} )f(\bx_p-r\underline{\omega} )+ (\underline{\omega}\partial_{r}  ) f(\bx_p+r\underline{\omega})  )\\
 &=&-\frac{1}{2}  \partial_{r} (   f(\bx_p-r\underline{\omega} )+    f(\bx_p+r\underline{\omega})  )\\
& =& -\partial_{r} F_1(\bx_p,r).
\end{eqnarray*}
Hence, $f \in\mathcal {GSR}(\Omega)$, as desired.
\end{proof}


We point out that, in the proof of  \cite[Theorem 4.5 (ii)]{Xu-Sabadini} i.e. the equality $\overline{D}_{\boldsymbol{x}_p}  F_2=- \partial_{r} F_1$, a minus sign was lost. The minus sign is correctly shown in the proof of the above Theorem \ref{relation-GSR-GSM}.

\section{Fueter polynomials and Taylor series expansion}
We now introduce some suitable polynomials which are generalized partial-slice monogenic and are the building blocks of the Taylor expansion of generalized  partial-slice monogenic functions. We start by giving the following definition.
\begin{definition}\label{Fueter-variables}
 The so-called (left) Fueter variables are defined as
 $$ z_{\ell}(\bx)=z_{\ell}^{L}(\bx)= z_{\ell}(\bx_{p}+r\underline{\omega})= x_{\ell}+r \underline{\omega} e_\ell, \ \ell=0,1,\ldots,p.$$
Similarly, the so-called right  Fueter variables are defined as
 $$ z_{\ell}^{R}(\bx)= z_{\ell}^{R}(\bx_{p}+r\underline{\omega})=  x_{\ell}+r e_\ell  \underline{\omega},\ \ell=0,1,\ldots,p.$$
 \end{definition}
An easy calculation using Proposition \ref{artin} shows that
$$z_{\ell}^{L}(\bx)\in \mathcal{GSM}^{L}(\mathbb{O}), \ z_{\ell}^{R}(\bx)\in \mathcal{GSM}^{R}(\mathbb{O}),$$
meanwhile,
$$z_{\ell}^{L}(\bx) D_{\underline{\omega}} =  D_{\underline{\omega}} z_{\ell}^{R}(\bx)=2e_\ell.$$
\begin{remark}{\rm
Note that, when $p=0$,  all left and right  Fueter variables coincide, i.e.
$$ z_{0}^{L}(\bx)= z_{0}^{L}( x_{0}+r\underline{\omega} ) = z_{0}^{R}( x_{0}+r\underline{\omega} )=\bx.$$
In this case,  $\bx a\in \mathcal{GSM}^{L}(\mathbb{O})$  for all $a \in \mathbb{O}$, while $(\bx a)b $ does not
necessarily belong to $\mathcal{GSM}^{L}(\mathbb{O})$ for  $b \in \mathbb{O}$.
}
\end{remark}
\begin{example}\label{za}
For $p\geq1$ and $ a\in \mathbb{O}$,  the function $z_{\ell}(\bx)a $ $ (\ell\geq1)$ does not necessarily belong to $\mathcal{GSM}^{L}$.
To see this,  we take  $a=e_{2}$ and $\underline{\omega}=e_7$,  then
$$D_{\underline{\omega}} (z_{1}(\bx)a)=e_1 e_2+\underline{\omega} ((\underline{\omega}e_1) e_2)=2e_3\neq0.$$
\end{example}

A natural idea is to use  Fueter variables in Definition \ref{Fueter-variables} to construct  generalized partial-slice monogenic Fueter polynomials as in the Cliffordian   case \cite{Xu-Sabadini}.  In view of the lack of associativity in $\mathbb{O}$, we need say more about the multiplication of  ordered $n$ elements.

Given an  alignment $(x_1, x_2,\ldots ,x_n)\in \mathbb{O}^{n}$, it is known that the multiplication of  ordered $n$ $(\geq2)$ elements
$x_1 x_2\cdots x_n$  has $ \frac{(2n-2)!}{n!(n-1)!}$  different associative orders.  Let
$(x_{1}  x_{2} \cdots  x_{n})_{\otimes_{n}}$ be the product of the ordered $n$ elements $(x_1, x_2,\ldots ,x_n)$ in a fixed associative order $\otimes_{n}$. In particular, denote the multiplication  from left to right by
$$(x_1 x_2 \cdots x_n)_{L}:=(\ldots((x_1 x_2)x_3) \cdots )x_n,$$
and the multiplication  from  right to left by
$$(x_1 x_2 \cdots x_n)_{R}:=x_1(  \cdots (x_{n-2}(x_{n-1}x_n))\ldots ).$$

\begin{proposition}\label{no-order}
Let $a, x_0,x_1,  \ldots ,x_p\in \mathbb{O} $ and $(j_1,j_2,\ldots, j_k) \in \{0,1,\ldots,p\}^{k}$,  repetitions being allowed. Then the following sum is independent of the chosen associative order $\otimes_{(k+1)}$
\begin{eqnarray}\label{sum}
  \sum_{(i_1,i_2, \ldots, i_k)\in \sigma} (x_{i_1}  x_{i_2} \cdots  x_{i_k}a)_{\otimes_{(k+1)}},
   \end{eqnarray}
where the sum runs over  all distinguishable permutations $\sigma$ of $(j_1,j_2,\ldots, j_k)$. \\
In particular, we have
$$\sum_{(i_1,i_2, \ldots, i_k)\in \sigma} (x_{i_1}  x_{i_2} \cdots  x_{i_k}a)_L
=\sum_{(i_1,i_2, \ldots, i_k)\in \sigma} (x_{i_1}  x_{i_2} \cdots  x_{i_k}a)_{R}.$$
\end{proposition}
\begin{proof}
Denote $ x=\sum_{i=1}^{k}t_{i}x_{j_i}\in \mathbb{O}$, where $ t_{i}\in \mathbb{R}, j_i \in \{0,1,\ldots,p\}, i=1,2,\ldots,k$.
Observe  that, for a fixed associative order $\otimes_{(k+1)}$ and  all $a\in \mathbb{O}$, the sum in (\ref{sum})
  is the coefficient of $k_0!k_2!\cdots k_p!t_{1}t_{2}\cdots t_{k}$ in $(\underbrace{x x \cdots x}_{k} a)_{\otimes_{(k+1)}}$, where $k_{\ell}$  is the appearing times of $\ell$ in $(j_1,j_2,\ldots, j_k)$, $\ell=0,1,\ldots,p$.

By Proposition  \ref{artin}, it holds that for any associative order $\otimes_{(k+1)}$
 $$(\underbrace{x x \cdots x}_{k} a)_{\otimes_{(k+1)}}=x^{k}a,$$
which means  the sum in (\ref{sum}) dees not depend on the associative order $\otimes_{(k+1)}$. The proof is complete.
\end{proof}

The case of  Proposition \ref{no-order} for $a=1$  had been obtained for Fueter variables in \cite{Liao-Li-11}  by induction.
Proposition \ref{no-order} allows to  construct  generalized partial-slice monogenic Fueter polynomials in the octonionic setting.

\begin{definition}\label{definition-Fueter}
For  $\mathrm{k}=(k_0,k_1,\ldots,k_{p})\in \mathbb{N}^{p+1}$, let   $\overrightarrow{\mathrm{k}}:=(j_1,j_2,\ldots, j_k)$   be an alignment with the number of $0$ in the alignment is $k_0$, the number of $1$ is $k_1$, and the number of $p$ is $k_{p}$, where $k=|\mathrm{k}|=k_0+k_1+\cdots+k_{p}, 0\leq j_1\leq j_2\leq \ldots\leq j_k\leq p$. Define
$$\mathcal{P}_{\mathrm{k}}(\bx)  =\mathcal{P}^{L}_{\mathrm{k}}(\bx) =\frac{1}{k!}  \sum_{(i_1,i_2, \ldots, i_k)\in \sigma(\overrightarrow{\mathrm{k}})} z_{i_1}  z_{i_2} \cdots  z_{i_k},$$
where the sum runs over the $\dfrac{k!}{\mathrm{k}!}$ different permutations $\sigma(\overrightarrow{\mathrm{k}})$ of $\overrightarrow{\mathrm{k}}$. When $\mathrm{k}=(0,\ldots,0)=\mathbf{0}$, we set $\mathcal{P}_{ \mathbf{0}}(\bx)=1$; when there is at least one negative component in $\mathrm{k}$, we set $\mathcal{P}_{\mathrm{k}}(\bx)=0$.

Similarly, we can define $\mathcal{P}^{R}_{  \mathrm{k}}$ when $z_{\ell}$ are replaced by $z_{\ell}^{R}$.
\end{definition}

\begin{remark}\label{Futer-special-case-p-0}{\rm
For $p=0$ and $\mathrm{k} =k\in \mathbb{N}$, we have by definition
$$\mathcal{P}^{L}_{\mathrm{k}}(\bx)=\mathcal{P}^{R}_{\mathrm{k}}(\bx)=\frac{1}{k!}\bx^{k}.$$
}
\end{remark}

In order to verify that $\mathcal{P}_{\mathrm{k}}$ belongs to $\mathcal{GSM}^{L}(\mathbb{O})$, we use a  Cauchy-Kovalevskaya extension result starting from some real analytic functions defined in a domain in $\mathbb{R}^{p+1}$.  For simplicity,  we consider here the case of polynomials defined in $\mathbb R^{p+1}$.
\begin{definition}[CK-extension]\label{Cauchy-Kovalevska-extension}
Let $f_{0}:\mathbb{R}^{p+1} \to \mathbb{O}$ be a  polynomial. Define the  generalized partial-slice Cauchy-Kovalevskaya extension (CK-extension, for short)  $CK[f_{0}]:  \mathbb{O}\to \mathbb{O}$ by
\begin{eqnarray}\label{series-CK}
 CK[f_{0}](\bx)=\sum_{k=0}^{+\infty} \frac{r^{2k}}{(2k)!}  (-\Delta_{\bx_p})^{k} f_{0}(\bx_p)
+\underline{\omega}  \sum_{k=0}^{+\infty} \frac{r^{2k+1}}{(2k+1)!}(-\Delta_{\bx_p}) ^{k}(D_{\bx_p} f_{0}(\bx_p)),
\end{eqnarray}
where $\bx=\bx_p+ r \underline{\omega}$ with $\bx_{p}\in\R^{p+1}, r \geq 0,$ and $\underline{\omega}  \in \mathbb S$.
\end{definition}

It should be pointed that $CK[f_{0}]$ is well-defined  since  the series in (\ref{series-CK})  does not depend  on $\underline{\omega}$ at the point $\bx_p$ and the series is in reality a finite sum when $f_0$ is a polynomial.

\begin{theorem}\label{CK-slice-monogenic}
Let $f_{0}:  \mathbb{R}^{p+1} \to \mathbb{O}$ be a  polynomial. Then $CK[f_0]$  is the unique extension  of $f_{0}$ to $\mathbb{O}$ which is generalized partial-slice monogenic.
\end{theorem}

\begin{proof}
From the Moufang identities and Proposition \ref{artin}, we have
$$D_{\bx_p} (\underline{\omega}  (D_{\bx_p} f_{0}))=(D_{\bx_p} \underline{\omega}   D_{\bx_p} )f_{0}=(\underline{\omega} \overline{D}_{\bx_p} D_{\bx_p} )f_{0}=\underline{\omega}\Delta_{\bx_p} f_{0}, $$
and
$$ \underline{\omega} (\underline{\omega} (D_{\bx_p} f_{0}) )= (\underline{\omega}\ \underline{\omega} )(D_{\bx_p} f_{0}) =-D_{\bx_p} f_{0},$$
which give
\begin{eqnarray*}
 & &D_{\underline{\omega}} CK[f_{0}](\bx_p+ r \underline{\omega})\\
 &=& (D_{\bx_p}+\underline{\omega}\partial_{r}) \Big(\sum_{k=0}^{+\infty} \frac{r^{2k}}{(2k)!}  (-\Delta_{\bx_p})^{k} f_{0}(\bx_p)\Big)\\
&& +(D_{\bx_p}+\underline{\omega}\partial_{r}) \Big( \underline{\omega}  \sum_{k=0}^{+\infty} \frac{r^{2k+1}}{(2k+1)!}(-\Delta_{\bx_p}) ^{k}(D_{\bx_p} f_{0}(\bx_p)) \Big)
 \\
&=&  \sum_{k=0}^{+\infty} \frac{r^{2k}}{(2k)!}  (-\Delta_{\bx_p})^{k}D_{\bx_p} f_{0}(\bx_p)+\underline{\omega}\sum_{k=1}^{+\infty} \frac{r^{2k-1}}{(2k-1)!}  (-\Delta_{\bx_p})^{k} f_{0}(\bx_p)
\\
& &+\underline{\omega}  \sum_{k=0}^{+\infty} \frac{r^{2k+1}}{(2k+1)!}(-\Delta_{\bx_p}) ^{k}  \Delta_{\bx_p} f_{0}(\bx_p)-
\sum_{k=0}^{+\infty} \frac{r^{2k}}{(2k)!}(-\Delta_{\bx_p}) ^{k}D_{\bx_p} f_{0}(\bx_p)
\\
&=&0,
\end{eqnarray*}
so we get $CK[f_{0}]\in \mathcal{GSM}^{L}(\mathbb{O})$.
Finally, Theorem \ref{Identity-theorem} gives the uniqueness of extension.
\end{proof}

\begin{remark}\label{CK-real-special-case}{\rm
If $f_{0}$ has values in an associative algebra, in particular if $f_0$ is  real-valued,  $CK[f_{0}]$ has the following  decomposition
\begin{eqnarray*}
  CK[f_{0}](\bx_p+ r \underline{\omega})&=&\Psi_{1}(\bx_{p},r)
+\underline{\omega}  \Psi_{2}(\bx_{p},r)\\
&=&(\partial_r+\underline{\omega} D_{\bx_p}) \sum_{k=0}^{+\infty} \frac{r^{2k+1}}{(2k+1)!}(-\Delta_{\bx_p}) ^{k}  f_{0}(\bx_p),
\end{eqnarray*}
where
 $$\Psi_{1}(\bx_{p},r)=\sum_{k=0}^{+\infty} \frac{r^{2k}}{(2k)!}  (-\Delta_{\bx_p})^{k} f_{0}(\bx_p)\in \mathbb{R}, $$
 $$\Psi_{2}(\bx_{p},r)=\sum_{k=0}^{+\infty} \frac{r^{2k+1}}{(2k+1)!}(-\Delta_{\bx_p}) ^{k} D_{\bx_p} f_{0}(\bx_p)\in \mathbb{ R}^{p+1}.$$
 }
\end{remark}

\begin{definition}\label{Cauchy-Kovalevska-extension-Fueter}
For  $\mathrm{k}=(k_0,k_1,\ldots,k_{p})\in\mathbb{N}^{p+1}$ and $\bx_{p}^{\mathrm{k}}=x_0^{k_0}\ldots x_p^{k_p}$, define
  $$V_{\mathrm{k}}(\bx)=\frac{1}{\mathrm{k}!}CK[\bx_{p}^{\mathrm{k}}](\bx),$$
 where $\mathrm{k}!=k_0!k_1!\cdots k_p!$. \\
 In particular, $z_{\ell}(\bx)=  CK[x_{\ell}](\bx)=x_\ell+r\underline{\omega} e_\ell,   \ell=0,1,\ldots,p,$
  where $\bx=\bx_p+ r \underline{\omega}$ with $\bx_{p}\in\R^{p+1}, r \geq 0,$ and $\underline{\omega}  \in \mathbb S$.

\end{definition}

Denote the \textit{commutator} by $[a,b]=ab-ba$ for $a,b\in \mathbb{O}.$
\begin{proposition}\label{V-k-relation}
 For each $\mathrm{k} \in \mathbb{N}^{p+1}$, there holds
 \begin{equation}\label{k-k-1}
\sum_{i=0}^p z_{i} V_{\mathrm{k}-\mathrm{\epsilon}_i}=|\mathrm{k}| V_{\mathrm{k}}=
 \sum_{i=0}^p V_{\mathrm{k}-\mathrm{\epsilon}_i} z_{i},
\end{equation}
 where $V_{\mathrm{k}-\mathrm{\epsilon}_i}=0$ if there is a negative integer in   $\mathrm{k}-\mathrm{\epsilon}_i$.
\end{proposition}
\begin{proof}
Fix $\underline{\omega}  \in \mathbb S$ and all functions  considered below shall be restricted to $\mathrm{H}_{\underline{\omega}}$.
In view of Remark \ref{CK-real-special-case}, we can set  $V_{\mathrm{k}}=\Phi_{\mathrm{k}} + \underline{\omega} \Psi_{\mathrm{k}} $, where $\Phi_{\mathrm{k}}\in \mathbb{R}, \Psi_{\mathrm{k} }=\sum_{i=0}^p \Psi_{\mathrm{k}, i} e_i\in \mathbb{R}^{p+1}$ with $  \Psi_{\mathrm{k}, i}\in \mathbb{R}$ for $i=0,1,\ldots,p$.
By definition, we have
\begin{equation}\label{V-i}
\partial_{x_i}V_{\mathrm{k}} =V_{\mathrm{k}-\mathrm{\epsilon}_i}, \quad 0\leq i\leq p,\end{equation}
which implies that
  \begin{equation}\label{V-i-j}
 \Psi_{\mathrm{k}-\mathrm{\epsilon}_i, j}=\Psi_{\mathrm{k}-\mathrm{\epsilon}_j, i}, \quad 0\leq i,j\leq p,
\end{equation}
and
\begin{equation}\label{V-i,j}
\partial_{x_j}  V_{\mathrm{k}-\mathrm{\epsilon}_i}=V_{\mathrm{k}-\mathrm{\epsilon}_i-\mathrm{\epsilon}_j}=V_{\mathrm{k}-\mathrm{\epsilon}_j-\mathrm{\epsilon}_i}=
\partial_{x_i}  V_{\mathrm{k}-\mathrm{\epsilon}_j}, \quad 0\leq i,j\leq p.\end{equation}
To prove the conclusion, we first prove four facts.\\
Fact 1:
$$\sum_{j=0}^p  \sum_{i=0}^p [z_i, \underline{\omega}e_j,  V_{\mathrm{k}-\mathrm{\epsilon}_i-\mathrm{\epsilon}_j}]=0.$$
To see this, we compute
 \[
   \begin{split}
  \sum_{j=0}^p  \sum_{i=0}^p [z_i, \underline{\omega}e_j,  V_{\mathrm{k}-\mathrm{\epsilon}_i-\mathrm{\epsilon}_j}]
   &= r\sum_{j=0}^p  \sum_{i=0}^p [\underline{\omega}e_i, \underline{\omega}e_j,  V_{\mathrm{k}-\mathrm{\epsilon}_i-\mathrm{\epsilon}_j}]
   \\
   &= r\sum_{0\leq i\neq j\leq p}[\underline{\omega}e_i, \underline{\omega}e_j,  V_{\mathrm{k}-\mathrm{\epsilon}_i-\mathrm{\epsilon}_j}]\\
    &= r\sum_{0\leq i< j\leq p} ([\underline{\omega}e_i, \underline{\omega}e_j,  V_{\mathrm{k}-\mathrm{\epsilon}_i-\mathrm{\epsilon}_j}]+[\underline{\omega}e_j, \underline{\omega}e_i,  V_{\mathrm{k}-\mathrm{\epsilon}_j-\mathrm{\epsilon}_i}])\\
  &= r\sum_{0\leq i< j\leq p} ([\underline{\omega}e_i, \underline{\omega}e_j,  V_{\mathrm{k}-\mathrm{\epsilon}_i-\mathrm{\epsilon}_j}]+[\underline{\omega}e_j, \underline{\omega}e_i,  V_{\mathrm{k}-\mathrm{\epsilon}_i-\mathrm{\epsilon}_j}])\\
  & =0.
  \end{split}
  \]
where the fourth  equality depends on (\ref{V-i,j}).\\
Fact 2:
$$\sum_{i=0}^p [\underline{\omega}, e_i,  V_{\mathrm{k}-\mathrm{\epsilon}_i} ]=0. $$
In fact, by Proposition \ref{artin} and Moufang identities we have the following chain of equalities
 \[
   \begin{split}
   \sum_{i=0}^p [\underline{\omega}, e_i,  V_{\mathrm{k}-\mathrm{\epsilon}_i} ]
   &= \sum_{i=1}^p \sum_{j=1}^p[\underline{\omega}, e_i,   \underline{\omega} e_j ]\Psi_{\mathrm{k}-\mathrm{\epsilon}_i, j}\\
  &=- \sum_{1\leq i\neq j\leq p}   [ e_i,\underline{\omega},  \underline{\omega} e_j ]\Psi_{\mathrm{k}-\mathrm{\epsilon}_i, j} \\
  &=
 - \sum_{1\leq i\neq j\leq p} ( (e_i \underline{\omega})(\underline{\omega} e_j)-e_i (\underline{\omega}(\underline{\omega} e_j)) )\Psi_{\mathrm{k}-\mathrm{\epsilon}_i, j}\\
  &= - \sum_{1\leq i\neq j\leq p} ( (\underline{\omega}e_i)( e_j\underline{\omega})+e_i  e_j )\Psi_{\mathrm{k}-\mathrm{\epsilon}_i, j}\\
  &=- \sum_{1\leq i\neq j\leq p} (  \underline{\omega} (e_i  e_j) \underline{\omega}+e_i  e_j )\Psi_{\mathrm{k}-\mathrm{\epsilon}_i, j}.
  \end{split}
  \]
Hence, using (\ref{V-i-j}) we have
$$\sum_{i=0}^p [\underline{\omega}, e_i,  V_{\mathrm{k}-\mathrm{\epsilon}_i} ]=- \sum_{1\leq i< j\leq p} (  \underline{\omega} (e_i  e_j+e_j  e_i) \underline{\omega}+e_i  e_j+e_j  e_i )\Psi_{\mathrm{k}-\mathrm{\epsilon}_i, j}=0,$$
and Fact 2 follows.
\\
 Fact 3: We prove that
 $$\sum_{i=0}^p  (\underline{\omega}e_i) V_{\mathrm{k}-\mathrm{\epsilon}_i}=\sum_{i=0}^p V_{\mathrm{k}-\mathrm{\epsilon}_i}  (\underline{\omega}e_i),$$
equivalently,
$$\sum_{i=0}^p  z_i V_{\mathrm{k}-\mathrm{\epsilon}_i}=\sum_{i=0}^p V_{\mathrm{k}-\mathrm{\epsilon}_i} z_i .$$
Here we only need to prove
\[
   \begin{split}
   \sum_{i=0}^p [\underline{\omega}e_i,  V_{\mathrm{k}-\mathrm{\epsilon}_i} ]
   &= \sum_{i=0}^p \sum_{j=0}^p[\underline{\omega}e_i,   \underline{\omega} e_j ]\Psi_{\mathrm{k}-\mathrm{\epsilon}_i, j}\\
  &= \sum_{0\leq i\neq j\leq p}   [ \underline{\omega}e_i,  \underline{\omega} e_j ]\Psi_{\mathrm{k}-\mathrm{\epsilon}_i, j} \\
  &=
  \sum_{0\leq i<j\leq p} ( [ \underline{\omega}e_i,  \underline{\omega} e_j ]+[ \underline{\omega}e_j,  \underline{\omega} e_i ] )\Psi_{\mathrm{k}-\mathrm{\epsilon}_i, j}= 0.
  \end{split}
  \]
Fact 4:
\begin{equation}\label{k-k-r}
\partial_{r} V_{\mathrm{k}} =\sum_{i=0}^p (\underline{\omega}e_{i}) V_{\mathrm{k}-\mathrm{\epsilon}_i}=
 \sum_{i=0}^p V_{\mathrm{k}-\mathrm{\epsilon}_i} (\underline{\omega}e_{i}).
\end{equation}
Recalling  Theorem \ref{CK-slice-monogenic}, we have $$(D_{\bx_p}   + \underline{\omega}   \partial_r ) V_{\mathrm{k}}=0,$$
which implies that by Proposition \ref{artin}, (\ref{V-i}), and Fact 2
 $$   \partial_r V_{\mathrm{k}}=  \underline{\omega} (D_{\bx_p} V_{\mathrm{k}})=  \underline{\omega}\sum_{j=0}^p  (e_j  V_{\mathrm{k}-\mathrm{\epsilon}_j}) =\sum_{j=0}^p (\underline{\omega}e_{j}) V_{\mathrm{k}-\mathrm{\epsilon}_j}.$$

Now we prove  the assertion  by induction on $k=|\mathrm{k}|$. The assertion  in (\ref{k-k-1}) is true for $k=0,1$. Assume that the assertion  in (\ref{k-k-1}) holds  for $k-1$, let us   prove it
for $k$. To this end,  we first observe that
$$V_{\mathrm{k}}(\bx_{p})=\frac{1}{\mathrm{k}!} \bx_{p}^{\mathrm{k}},$$
and then
$$(\sum_{i=0}^p z_{i} V_{\mathrm{k}-\mathrm{\epsilon}_i} )(\bx_{p})=\sum_{i=0}^p x_{i} V_{\mathrm{k}-\mathrm{\epsilon}_i}(\bx_{p})
=\sum_{i=0}^p\frac{ x_{i} }{(\mathrm{k-\mathrm{\epsilon}_i})!} \bx_{p}^{\mathrm{k}-\mathrm{\epsilon}_i}= |\mathrm{k}|  \frac{1}{\mathrm{k}!} \bx_{p}^{\mathrm{k}}=|\mathrm{k}| V_{\mathrm{k}}(\bx_{p}).$$
Hence, by Theorems \ref{Identity-theorem} and  \ref{CK-slice-monogenic} and Fact 3,  it remains to show   $\sum_{i=0}^p z_{i} V_{\mathrm{k}-\mathrm{\epsilon}_i}\in \mathcal{GSM}^{L}(\mathbb O)$.  From (\ref{V-i}) and (\ref{V-i,j}), it holds that for $j=0,1,\ldots,p,$
\[
\begin{split}
\partial_{x_j} \sum_{i=0}^p z_{i} V_{\mathrm{k}-\mathrm{\epsilon}_i}
&=  \sum_{i=0}^p  ( \delta_{ij} V_{\mathrm{k}-\mathrm{\epsilon}_i} +z_i V_{\mathrm{k}-\mathrm{\epsilon}_i-\mathrm{\epsilon}_j}) \\
 &=   V_{\mathrm{k}-\mathrm{\epsilon}_j}+ (|\mathrm{k}| -1)    V_{\mathrm{k}-\mathrm{\epsilon}_j} \\
&=|\mathrm{k}|      V_{\mathrm{k}-\mathrm{\epsilon}_j} =|\mathrm{k}|    \partial_{x_j} V_{\mathrm{k}},
\end{split}
\]
which gives
 \begin{equation}\label{V-i,j-D-xp}
 D_{\bx_p} \sum_{i=0}^p z_{i} V_{\mathrm{k}-\mathrm{\epsilon}_i}
 =|\mathrm{k}|  \sum_{i=0}^p e_i   \partial_{x_i} V_{\mathrm{k}} =|\mathrm{k}| D_{\bx_p} V_{\mathrm{k}}.\end{equation}
Furthermore, it holds that
\[
\begin{split}
  \partial_{r}\sum_{i=0}^p z_{i} V_{\mathrm{k}-\mathrm{\epsilon}_i}
&=   \sum_{i=0}^p  ((\underline{\omega}e_i) V_{\mathrm{k}-\mathrm{\epsilon}_i} + z_i \partial_{r}V_{\mathrm{k}-\mathrm{\epsilon}_i})\\
&=   \sum_{i=0}^p  (\underline{\omega}e_i) V_{\mathrm{k}-\mathrm{\epsilon}_i} +
 \sum_{i=0}^p z_i( \sum_{j=0}^p   (\underline{\omega}e_j)  V_{\mathrm{k}-\mathrm{\epsilon}_i-\mathrm{\epsilon}_j})\\
&= \sum_{i=0}^p  (\underline{\omega}e_i) V_{\mathrm{k}-\mathrm{\epsilon}_i}+  \sum_{i=0}^p z_i( \sum_{j=0}^p   V_{\mathrm{k}-\mathrm{\epsilon}_i-\mathrm{\epsilon}_j}(\underline{\omega}e_j))\\
&=\sum_{i=0}^p  (\underline{\omega}e_i) V_{\mathrm{k}-\mathrm{\epsilon}_i}+
 \sum_{j=0}^p\sum_{i=0}^p ( z_i  V_{\mathrm{k}-\mathrm{\epsilon}_i-\mathrm{\epsilon}_j})(\underline{\omega}e_j)\\
&=\sum_{i=0}^p  (\underline{\omega}e_i) V_{\mathrm{k}-\mathrm{\epsilon}_i}+
(|\mathrm{k}| -1)\sum_{j=0}^p   V_{\mathrm{k}-\mathrm{\epsilon}_j}(\underline{\omega}e_j)\\
 &=|\mathrm{k}|  \sum_{j=0}^p  (\underline{\omega}e_j) V_{\mathrm{k}-\mathrm{\epsilon}_j}\\
  &=|\mathrm{k}|  \partial_{r}  V_{\mathrm{k}},
 \end{split}
\]
where the second,  third, fourth, sixth, and    final equalities  depends on Fact 4, Fact 3,  Fact 1, Fact 3, and  Fact 4, respectively.
Combining this with (\ref{V-i,j-D-xp}),  we finally  get
  $$ D_{ \underline{\omega}} \sum_{i=0}^p z_{i} V_{\mathrm{k}-\mathrm{\epsilon}_i}=|\mathrm{k}|  D_{ \underline{\omega}} V_{\mathrm{k}}=0,$$
which finishes the proof.

\end{proof}

\begin{theorem}\label{V-P}
 For each $\mathrm{k} \in \mathbb{N}^{p+1}$, there holds
 $ V_{\mathrm{k}}(\bx)=  P_{\mathrm{k}}(\bx).$
  \end{theorem}
\begin{proof}
By Propositions \ref{V-k-relation} and \ref{no-order}, we have for $\mathrm{k} \in \mathbb{N}^{p+1}$
 \begin{eqnarray*}
 V_{\mathrm{k}}(\bx)&=&\frac{1}{|\mathrm{k}|}\sum_{i=0}^p z_{i} V_{\mathrm{k}-\mathrm{\epsilon}_i}(\bx)\\
&=&\frac{1}{|\mathrm{k}|(|\mathrm{k}|-1)}\sum_{i,j=0}^p z_{i} (z_j V_{\mathrm{k}-\mathrm{\epsilon}_i -\mathrm{\epsilon}_j} (\bx) )
 \\
&=&  \cdots \\
   &= & \frac{1}{|\mathrm{k}|! }  \sum_{i_1,\ldots, i_k=0}^p (z
   _{i_1}\cdots z_{i_k}V_{\mathrm{k}-\mathrm{\epsilon}_{i_i}-\cdots-\mathrm{\epsilon}_{i_k}}(\bx))_{R}\\
  &=& \frac{1}{|\mathrm{k}|! } \sum_{\mathrm{k}=\mathrm{\epsilon}_{i_1}+\cdots+\mathrm{\epsilon}_{i_k}} (z_{i_1}\cdots z_{i_k} )_{R}\\
   &=& P_{\mathrm{k}}(\bx),
\end{eqnarray*}
which completes the proof.
\end{proof}


\begin{lemma}\label{Taylor-E-lemma}
Let $\underline{\eta}\in \mathbb{S}$ and   $\Omega\subseteq \mathbb{O}$ be a   domain. If $\phi \in C^1(\Omega, \mathbb{O}) $  is of the form
\begin{equation}\label{slice-form-E-Taylor}
\phi(\bx_{p}+r\underline{\eta})=\Phi(\bx')+\underline{\eta} \Psi(\bx'),
\end{equation}
 where $\Psi(\bx') \in \mathbb{R}$, $ \Phi(\bx')=\sum_{i=0}^p \Phi_{i}(\bx') e_i\in \mathbb{R}^{p+1}$ with $  \Phi_{i}\in \mathbb{R}, i=0,1,\ldots,p$, and satisfies
 \begin{equation}\label{part-p-i-j}
\partial_{x_i}\Phi_j= \partial_{x_j}\Phi_i, \quad 1\leq i,j\leq p,
\end{equation}
 then
$$\sum_{i=0}^{p}[\underline{\eta},e_i,    \partial_{x_i}\phi_{\underline{\eta}} ]=0.$$
\end{lemma}
\begin{proof}
By direct calculations, we have
 \[
   \begin{split}
   \sum_{i=0}^{p} [\underline{\eta},e_i,     \partial_{x_i} \phi_{\underline{\eta}} ]
   &= \sum_{i,j=1}^{p} [\underline{\eta},e_i,   e_j ]  \partial_{x_i} \Phi_j\\
  &= \sum_{1\leq i\neq j\leq p} ( [\underline{\eta},e_i,   e_j ] +  [\underline{\eta},e_j,   e_i] ) \partial_{x_i} \Phi_j=0,
  \end{split}
  \]
  as desired.
\end{proof}
\begin{lemma}\label{lemmakth-E}
Let $\underline{\eta}\in\mathbb S$ and consider the polynomial $P$ defined  in $\mathbb{O}$   satisfying that its restriction $P_{\underline{\eta}}$ is homogeneous  of degree $k$ with
$$(\sum_{i=0}^{p}e_i\partial_{x_i}+\underline{\eta}\partial_{r})P(\bx_p+\underline{\eta} r)=0,$$
and
\begin{equation}\label{polynomial-associative}
\sum_{i=0}^{p}[\underline{\eta},e_i,     \partial_{x_i} P_{\underline{\eta}} ]=0.\end{equation}
Then we have
 $$P(\bx_p+\underline{\eta} r)= \sum_{|\mathrm{k}|=k} \mathcal{P}_{\mathrm{k}}  (\bx_p+r\underline{\eta}) a_{\mathrm{k}},\quad  a_{\mathrm{k}}= \partial_{\mathrm{k}} P(0). $$
\end{lemma}
\begin{proof}
Fix $\underline{\eta}\in\mathbb S$ and consider a generic polynomial, homogeneous of degree $k$,
 satisfying (\ref{polynomial-associative}) and
 $$
 (\sum_{i=0}^{p}e_i\partial_{x_i}+\underline{\eta}\partial_{r})P(\bx_p+\underline{\eta} r)=0,
 $$
 from which we deduce   by Proposition \ref{artin}
 \begin{equation}\label{Poly}
 \partial_{r}P(\bx_p+\underline{\eta} r)=\underline{\eta} (\sum_{i=0}^{p} e_i\partial_{x_i} P(\bx_p+\underline{\eta} r))
 =\sum_{i=0}^{p} (\underline{\eta}e_i)\partial_{x_i} P(\bx_p+\underline{\eta} r)  .
 \end{equation}
 Since $P_{\underline{\eta}}$ is homogeneous of degree $k$, $P_{\underline{\eta}}$ satisfies also
 $$
 \sum_{i=0}^p x_i\partial_{x_i} P_{\underline{\eta}} + r \partial_r P_{\underline{\eta}}= k P_{\underline{\eta}},
 $$
 and substituting \eqref{Poly} in this last expression we get
 $$
 k P(\bx_p+\underline{\eta} r)=  \sum_{i=0}^p (x_i+\underline{\eta} e_i r) \partial_{x_i}P(\bx_p+\underline{\eta} r)= \sum_{i=0}^p z_i  \partial_{x_i}P(\bx_p+\underline{\eta} r).
 $$
Now we iterate the procedure for the derivatives $\partial_{x_i}P_{\underline{\eta}}$, $i=0,\,\ldots, p,$ which are homogeneous polynomials of degree $(k-1)$, in the kernel of the operator $(D_{\bx_p}+\underline{\eta}\partial r)$,  and
 satisfy  (\ref{polynomial-associative}).  After $k$ iterations, we get
  $$
  k! P(\bx_p+\underline{\eta} r)=  \sum_{i_1,\ldots, i_k=0}^p \Big(z_{i_1}\ldots  z_{i_k}\frac{\partial^k}{\partial_{x_{i_1}}\ldots \partial_{x_{i_k}}}P(\bx_p+\underline{\eta} r) \Big)_{R}.
  $$
  Since the order of derivation does not affect the calculations, we can group all the derivatives of the form
  $\partial_{\mathrm{k}}=\frac{\partial^{k} P }{\partial_{x_{0}}^{k_0}\partial_{x_{1}}^{k_1}\cdots\partial_{x_{p}}^{k_p} }$ with $\mathrm{k}=(k_0,k_1,\ldots,k_{p})$ and $k=|\mathrm{k}|$:
  \[
  \begin{split}
  P(\bx_p+\underline{\eta} r)&=  \frac{1}{k! }  \sum_{i_1,\ldots, i_k=0}^p \Big(z_{i_1}\ldots z_{i_k}\frac{\partial^k}{\partial_{x_{i_1}}\ldots \partial_{x_{i_k}}}P(\bx_p+\underline{\eta} r) \Big)_{R}\\
  &=\frac{1}{k! } \sum_{|\mathrm{k}|=k}  \sum_{(i_1,i_2, \ldots, i_k)\in \sigma(\overrightarrow{\mathrm{k}})  }
  (z_{i_1}  z_{i_2} \cdots  z_{i_k}  \partial_{\mathrm{k}} P(\bx_p+\underline{\eta} r))_{R}\\
  &=\frac{1}{k! } \sum_{|\mathrm{k}|=k}  \sum_{(i_1,i_2, \ldots, i_k)\in \sigma(\overrightarrow{\mathrm{k}})  }
  (z_{i_1}  z_{i_2} \cdots  z_{i_k}  \partial_{\mathrm{k}} P(\bx_p+\underline{\eta} r))_{L}\\
   &=\sum_{|\mathrm{k}|=k} \mathcal{P}_{\mathrm{k}}(\bx_p+\underline{\eta} r)\partial_{\mathrm{k}}P(\bx_p+\underline{\eta} r)\\
  &=\sum_{|\mathrm{k}|=k} \mathcal{P}_{\mathrm{k}}(\bx_p+\underline{\eta} r) \partial_{\mathrm{k}} P(0),
  \end{split}
  \]
  where $\sigma(\overrightarrow{\mathrm{k}})$ is  as in Definition \ref{definition-Fueter}, the third equality follows from Proposition \ref{no-order} and last equality follows from the fact that $P_{\underline{\eta}}$ has degree $k$.
  The proof is complete.
\end{proof}
For   $\mathrm{k} \in \mathbb{N}^{p+1}$, define
$$\mathcal{Q}_{ \mathrm{k}}  (\bx):=(-1)^{|\mathrm{k}|}  \partial_{\mathrm{k}} E(\bx).$$
In particular, for $p=0$ and $\mathrm{k} =k\in \mathbb{N}$, we have by definition
$$\mathcal{Q}_{ \mathrm{k}}  (\bx)= \frac{k!}{2\pi} \bx^{-(k+1)}.$$

\begin{lemma} \label{Taylor-lemma-E}
Given $\by\in \mathrm{H}_{\underline{\eta}}$ for some $\underline{\eta} \in \mathbb{S}$, it holds that for all $\bx\in B(|\by|)=\{\bx \in \mathrm{H}_{\underline{\eta}}: |\bx|<|\by|\}$
$$E_{\by}(\bx)=  \sum_{k=0}^{+\infty}\Big( \sum_{|\mathrm{k}|=k} \mathcal{P}_{\mathrm{k}}  (\bx) \mathcal{Q}_{ \mathrm{k}}  (\by)\Big)
=  \sum_{k=0}^{+\infty} \Big( \sum_{|\mathrm{k}|=k} \mathcal{Q}_{ \mathrm{k}}  (\by)\mathcal{P}^{R}_{\mathrm{k}}  (\bx) \Big), $$
where the series converges uniformly on any compact subsets of $ B(|\by|)$.
\end{lemma}
\begin{proof}
Here we only prove the first series expansion, the second can be obtained with the same strategy.
Let  $\bx,\by\in \mathrm{H}_{\underline{\eta}}$ with $|\bx|<|\by|$.

For $p=0$, it holds that
$$\frac{\overline{\by-\bx}}{|\by-\bx|^{2}}=\frac{1}{\by-\bx }=\sum_{k=0}^{+\infty}\bx^{k}\by^{-(k+1)}.$$
Hence, by Remark \ref{Futer-special-case-p-0}, we have
$$E_{\by}(\bx)=  \sum_{k=0}^{+\infty}  \mathcal{P}_{k}  (\bx) \mathcal{Q}_{k}  (\by). $$
which converges uniformly on any compact subsets of $ B(|\by|)$ since
$$\sum_{k=0}^{+\infty}|\bx^{k}\by^{-(k+1)}|=\sum_{k=0}^{+\infty}|\bx|^{k}|\by|^{-(k+1)} =\frac{1}{|\by|-|\bx|}. $$

For $p\geq 1$,  it holds that
$$ \frac{\overline{\by-\bx}}{|\by-\bx|^{p+2}} =\frac{1}{p } \overline{D}_{\bx,\underline{\eta}} \frac{1}{|\by-\bx|^{p}},$$
where $\bx=\bx_{p}+r\underline{\eta}, D_{\bx,\underline{\eta}}=D_{\bx_p}+\underline{\eta}\partial_{r}$.\\
From the formula
$$ \frac{1}{|\by-\bx|^{p}}=\sum_{k=0}^{\infty} \frac{(-1)^{k}}{k!} \langle \bx', \nabla_{\by'}\rangle^{k}   \frac{1}{|\by|^{p}},$$
where $\by'=(\by_{p}, s), s=|\underline{\by}_{q}|, \langle \bx', \nabla_{\by'}\rangle = \sum_{i=0}^{p}x_i\partial_{y_i} + r \partial_{s}$,\\
  we have
$$ \frac{\overline{\by-\bx}}{|\by-\bx|^{p+2}} =\sum_{k=0}^{\infty}P_{k}(\bx, \by), $$
where the homogeneous polynomial  $P_{k}( \cdot, \by)$  of degree $k$ ia given by
\begin{equation}\label{p-k-form-1}
P_{k}(\bx, \by)= \frac{(-1)^{k}}{k! p}\overline{D}_{\bx, \underline{\eta}} \big( \langle \bx', \nabla_{\by'} \rangle^{k}   \frac{1}{|\by|^{p}}\big)
\in ker D_{\bx,  \underline{\eta}},
\end{equation}
or
\begin{equation}\label{p-k-form-2}
P_{k}(\bx, \by)=- \frac{(-1)^{k}}{k! p} \overline{D}_{\by, \underline{\eta}} \big(\langle \bx', \nabla_{\by'} \rangle^{k} \frac{1}{|\by|^{p}}\big)
\in ker D_{\bx,  \underline{\eta}}. \end{equation}
Observe that  $P_{k}( \cdot, \by)$ in the form of (\ref{p-k-form-1}) satisfies the conditions in (\ref{slice-form-E-Taylor}) and  (\ref{part-p-i-j}), which gives by Lemma \ref{Taylor-E-lemma} that
\begin{equation}\label{lemmakth-E-condition}
\sum_{i=0}^{p}[\underline{\eta},e_i,    \partial_{x_i} P_{k}(\bx_p+r\underline{\eta}, \by) ]=0.\end{equation}
Notice  that  $P_{k}( \cdot, \by)$ in   (\ref{p-k-form-2}) takes also the form of
\begin{equation*}
P_{k}(\bx, \by)=- \frac{(-1)^{k}}{k! p} \langle \bx', \nabla_{\by'} \rangle^{k}\overline{D}_{\by, \underline{\eta}} \big( \frac{1}{|\by|^{p}}\big)
= \frac{(-1)^{k}}{k! }  \langle \bx', \nabla_{\by'} \rangle^{k}    \frac{\overline{\by}}{|\by|^{p+2}}, \end{equation*}
and we can prove by induction that, for all $\mathrm{k} \in \mathbb{N}^{p+1}$ with $|\mathrm{k}|=k$ and the function $h$ smooth  enough,
$$ \partial_{\bx, \mathrm{k}}(\langle \bx', \nabla_{\by'} \rangle^{k} h(\by))|_{ \bx'=0}= |\mathrm{k}|! \partial_{\by, \mathrm{k}} h(\by),$$
which imply that
\begin{equation}\label{lemmakth-E-zero}
\partial_{\bx, \mathrm{k}} P_{k}(0, \by)=(-1)^{k}   \partial_{\by, \mathrm{k}}   \frac{\overline{\by}}{|\by|^{p+2}}=
\sigma_{p+1} \mathcal{Q}_{ \mathrm{k}}  (\by).
\end{equation}
Consequently, in view of (\ref{lemmakth-E-condition}) and (\ref{lemmakth-E-zero}), we get the conclusion follows directly from Lemma \ref{lemmakth-E}:
\begin{equation}\label{left-Taylor}
E_{\by}(\bx)=  \sum_{k=0}^{+\infty}\Big( \sum_{|\mathrm{k}|=k} \mathcal{P}_{\mathrm{k}}  (\bx) \mathcal{Q}_{ \mathrm{k}}  (\by)\Big).
\end{equation}

 Now it remains to show the series in (\ref{left-Taylor}) converges uniformly on any compact subsets of $ B(|\by|)$. To this see, we consider  $\bx, \by \in \mathrm{H}_{\underline{\eta}}$ with $|\bx|<|\by|$, and then
$$ \frac{\overline{\by-\bx}}{|\by-\bx|^{p+2}} =
\sum_{k=0}^{\infty}P^{(k)}(\by^{-1}\bx)  \frac{\overline{\by}}{|\by|^{p+2}}, $$
where
$$P^{(k)}(\by^{-1}\bx)= \frac{|\bx|^{k}}{|\by|^{k}} C_{p+2,k}^{+}(\boldsymbol{\alpha},\boldsymbol{\beta}),$$
with $\bx=|\bx| \boldsymbol{\alpha}, \by=|\by|\boldsymbol{\beta}$, and
$$C_{p+2,k}^{+}(\boldsymbol{\alpha},\boldsymbol{\beta}) = \frac{ p+k}{p} C^{\frac{p}{2}}_{k}(\langle \boldsymbol{\alpha}',\boldsymbol{\beta}'\rangle)+
(\langle \boldsymbol{\alpha}',\boldsymbol{\beta}'\rangle- \overline{\boldsymbol{\alpha}} \boldsymbol{\beta})C^{\frac{p}{2}+1}_{k-1}(\langle \boldsymbol{\alpha}',\boldsymbol{\beta}'\rangle),$$
where $C_k^{\nu}$ is the Gegenbauer polynomial of degree $k$ associated with $\nu$.
This result can be proved as in the Clifford case and  we omit its details here, see e.g. \cite[p. 179-183]{Delanghe-Sommen-Soucek-92},  due to that  the calculations do  not involve associativity.

Note that $ P^{(k)}(\by^{-1}\bx)$ is  a polynomial in $\bx'$ of degree $k$ and satisfies the estimate \cite{Sommen-81}

$$|P^{(k)}(\by^{-1}\bx)|\leq C_{p} k^{p+1} \frac{|\bx|^{k}}{|\by|^{k}},$$
where $C_p$ stands for a constant depending on  $p$.

In view of the uniqueness of the Taylor expansion of real analytic functions, we have
$$P_{k}(\bx, \by)=P^{(k)}(\by^{-1}\bx)  \frac{\overline{\by}}{|\by|^{p+2}}.$$
Hence, it follows that
$$|P_{k}(\bx, \by)|\leq C_{p} k^{p+1} \frac{|\bx|^{k}}{|\by|^{k+p+1}},$$
which implies that the series in (\ref{left-Taylor})  converges uniformly on any compact subsets of $ B(|\by|)$. The proof is complete.
\end{proof}

Finally,  we can establish a  Taylor series expansion with a tail, which vanishes  in the case of Clifford (associative) algebras, see \cite[Theorem 3.28]{Xu-Sabadini}, or in the  slice monogenic case (namely $p=0$), \cite[Theorem 2.12]{Gentili-Struppa-10}.
\begin{theorem} \label{Taylor-lemma-T}
Let  $f:B(\rho) \rightarrow \mathbb{O}$ be a  generalized partial-slice monogenic function, where $B(\rho)=\{\bx \in \mathbb{O}: |\bx|<\rho\}$.
Then,   for all $x\in B(r)_{\underline{\eta}}$ with   $r<\rho$ and $\underline{\eta}\in \mathbb{S}$,
$$f(\bx)=  \sum_{k=0}^{+\infty} \Big( \sum_{|\mathrm{k}|=k} \mathcal{P}_{\mathrm{k}}  (\bx)\partial_{ \mathrm{k}}  f(0)+T_{\mathrm{k}}(\bx)\Big) , \quad \mathrm{k}=(k_0,k_1,\ldots,k_{p}),$$
where
$$T_{\mathrm{k}}(\bx)=\int_{\partial B(r)_{\underline{\eta}}}   [\mathcal{P}_{\mathrm{k}}  (\bx),  \mathcal{Q}_{ \mathrm{k}}(\by), \bn(\by)f(\by)] dS(\by),$$
 with $\bn(\by)=\by/r$  being the unit exterior normal to $\partial B(r)_{\underline{\eta}}$ at $\by$ and
 $dS$  being    the   classical   Lebesgue surface  element  in $\mathbb{R}^{p+2}$.
\end{theorem}
\begin{proof}
Let $f\in\mathcal {GSM}(B(\rho))$. For    $ \bx\in B(r)_{\underline{\eta}}$ with  $ r<\rho$ and $ \underline{\eta}\in \mathbb{S}$,   we have by Theorem \ref{Cauchy-slice}
\begin{equation}\label{Cauchy-1}
f(\bx)=\int_{\partial B(r)_{\underline{\eta}}}  E_{\by}(\bx) (\bn(\by)f(\by)) dS(\by),
 \end{equation}
where  $\bn(\by)=\by/r$ is the unit exterior normal to $\partial B(r)_{\underline{\eta}}$ at $\by$,
 $dS$  stands  for  the   classical   Lebesgue surface  element  in $\mathbb{R}^{p+2}$.\\
Hence, for $\mathrm{k}=(k_0,k_1,\ldots,k_{p})\in\mathbb{N}^{p+1}$,
$$\partial_{\mathrm{k}}  f(0)=\int_{\partial B(r)_{\underline{\eta}}}  \mathcal{Q}_{ \mathrm{k}}  (\by) (\bn(\by)f(\by)) dS(\by).$$
Recalling Lemma \ref{Taylor-lemma-E} and (\ref{Cauchy-1}), we get
\begin{eqnarray*}
 f(\bx) &=&\int_{\partial B(r)_{\underline{\eta}}}\Big(  \sum_{k=0}^{+\infty}  \Big( \sum_{|\mathrm{k}|=k} \mathcal{P}_{\mathrm{k}}  (\bx) \mathcal{Q}_{ \mathrm{k}}  (\by) \Big) \Big)(\bn(\by)f(\by)) dS(\by)\\
&=&  \int_{\partial B(r)_{\underline{\eta}}}   \sum_{k=0}^{+\infty} \Big(  \Big( \sum_{|\mathrm{k}|=k} \mathcal{P}_{\mathrm{k}}  (\bx) \mathcal{Q}_{ \mathrm{k}}  (\by) \Big) (\bn(\by)f(\by)) \Big)  dS(\by)
 \\
&=&  \sum_{k=0}^{+\infty}   \int_{\partial B(r)_{\underline{\eta}}}  \Big( \sum_{|\mathrm{k}|=k} \mathcal{P}_{\mathrm{k}}  (\bx) \mathcal{Q}_{ \mathrm{k}}  (\by)  \Big)  (\bn(\by)f(\by))  dS(\by)
\\
&=&\sum_{k=0}^{+\infty} \int_{\partial B(r)_{\underline{\eta}}}   \sum_{|\mathrm{k}|=k}( \mathcal{P}_{\mathrm{k}}  (\bx) \mathcal{Q}_{ \mathrm{k}}  (\by)  ) (\bn(\by)f(\by)) dS(\by)
\\
&=&\sum_{k=0}^{+\infty} \int_{\partial B(r)_{\underline{\eta}}} \sum_{|\mathrm{k}|=k} \Big( \mathcal{P}_{\mathrm{k}}  (\bx)( \mathcal{Q}_{ \mathrm{k}}  (\by) (\bn(\by)f(\by)))\\
 &&\ \ \ \ \ \ \  +[\mathcal{P}_{\mathrm{k}}  (\bx), \mathcal{Q}_{ \mathrm{k}}  (\by), \bn(\by)f(\by)]\Big) dS(\by)\\
  &=& \sum_{k=0}^{+\infty}\Big( \sum_{|\mathrm{k}|=k} \mathcal{P}_{\mathrm{k}}  (\bx) \int_{\partial B(r)_{\underline{\eta}}} \mathcal{Q}_{ \mathrm{k}}  (\by) (\bn(\by)f(\by))) dS(\by)  +T_{\mathrm{k}}(\bx)\Big)\\
   &=&\sum_{k=0}^{+\infty}\Big( \sum_{|\mathrm{k}|=k} \mathcal{P}_{\mathrm{k}}  (\bx)\partial_{ \mathrm{k}}  f(0)+T_{\mathrm{k}}(\bx)\Big),
\end{eqnarray*}
where
$$T_{\mathrm{k}}(\bx)=\int_{\partial B(r)_{\underline{\eta}}}   [\mathcal{P}_{\mathrm{k}}  (\bx),  \mathcal{Q}_{ \mathrm{k}}(\by), \bn(\by)f(\by)]  dS(\by).$$
The proof is complete.
\end{proof}
\section*{Final Conclusion}
In this paper we start the study of generalized partial-slice monogenic functions in a non-associative case.
We work in the octonionic framework, however  most of obtained  results only  rely on some   properties of real alternative algebras, such as the Artin theorem or  Moufang identities. Therefore, this paper could  be rewritten in the general context of real alternative $\ast$-algebras with the suitable changes in the notations and terminology. Moreover, in an alternative $\ast$-algebra, the domain of  generalized partial-slice monogenic  functions should be defined not in the full algebra but, in general, in a suitable hypercomplex subset. We obtain a number of results, among which the Representation Formula, the Cauchy (and Cauchy-Pompeiu) integral formula, the maximum modulus principle. We also study the analog of Fueter polynomials and the Taylor series expansion which differ from the one in the associative case, indeed  a tail appears in each summand which is typical of the non-associative case.



\bibliographystyle{plain}

\begin{thebibliography}{99}
 \bibitem{Baez}J. C. Baez,   \textit{The octonions}, Bull. Amer. Math. Soc. (N.S.) 39 (2002), no. 2, 145-205.

\bibitem{Brackx} F. Brackx, R. Delanghe, F. Sommen, \textit{Clifford analysis}, Research Notes in Mathematics, Vol. 76, Pitman, Boston, 1982.
\bibitem{Colombo-Krausshar-Sabadini-24} F. Colombo, R. S. Krau{\ss}har,   I. Sabadini,  \textit{Octonionic monogenic and slice monogenic Hardy and Bergman spaces}, Forum Math. 36 (2024), no. 4, 1031-1052.

    \bibitem{Colombo-Sabadini-Struppa-20} F. Colombo, I. Sabadini, D. C. Struppa, \textit{Michele Sce's Works in Hypercomplex Analysis. A Translation with Commentaries}, Birkh\"auser, Basel, 2020.

   \bibitem{Colombo-Sabadini-Struppa-00} F. Colombo, I. Sabadini, D. C. Struppa, \textit{Dirac equation in the octonionic algebra},  Contemp. Math. 251 (2000), 117-134.

\bibitem{Colombo-Sabadini-Struppa-09} F. Colombo, I. Sabadini, D. C. Struppa, \textit{Slice monogenic functions}, Israel J. Math.  171 (2009), 385-403.

\bibitem{Colombo-Sabadini-Struppa-11} F. Colombo, I. Sabadini, D. C. Struppa,
\textit{Noncommutative functional calculus. Theory and applications of slice hyperholomorphic functions},
Volume 289, Progress in Mathematics, Birkh\"auser, Basel (2011).

\bibitem{Delanghe-Sommen-Soucek-92}R. Delanghe, F. Sommen, V. Vladim\'{\i}r Sou\v{c}ek, \textit{Clifford algebra and spinor-valued functions. A function theory for the Dirac operator}, Mathematics and its Applications, Vol. 53, Kluwer Academic Publishers Group, Dordrecht, 1992.
 \bibitem{Dentoni-Sce} P. Dentoni, M. Sce, \textit{Funzioni regolari nell'algebra di Cayley},   Rend. Sem. Mat. Univ. Padova 50 (1973), 251-267.
 \bibitem{DingXu}    C. Ding, Z. Xu, {\it Invariance of iterated global differential operator for slice monogenic
functions},  Comput. Methods Funct. Theory 25 (2025), no. 3, 735-752.

\bibitem{Dou-23}  X.  Dou, G. Ren, I. Sabadini,   \textit{A representation formula for slice regular functions over slice-cones in several variables}, Ann. Mat. Pura Appl. (4) 202 (2023), no. 5, 2421-2446.
\bibitem{Dou}  X.  Dou, G. Ren, I. Sabadini, T. Yang,  \textit{Weak slice regular functions on the $n$-dimensional quadratic cone of octonions}, J. Geom. Anal. 31 (2021), no. 11, 11312-11337.


  \bibitem{Gentili-Struppa-07} G.  Gentili, D. C. Struppa, \textit{A new theory of regular functions of a quaternionic variable}, Adv. Math. 216 (2007) no. 1, 279-301.

     \bibitem{Gentili-Struppa-10} G. Gentili, D. C. Struppa, \textit{Regular functions on the space of Cayley numbers}, Rocky Mountain J. Math. 40 (2010), no. 1, 225-241.

      \bibitem{Ghiloni-Perotti} R.  Ghiloni,  A. Perotti, \textit{Slice regular functions on real alternative algebras}, Adv. Math. 226 (2011), no. 2, 1662-1691.
  \bibitem{Ghiloni-Perotti-Stoppato-20} R. Ghiloni, A. Perotti, C. Stoppato, \textit{Division algebras of slice functions}, Proc. Roy. Soc. Edinburgh Sect. A 150 (2020), no. 4, 2055-2082

       \bibitem{Ghiloni-Stoppato-24-1} R. Ghiloni, C. Stoppato, \textit{A unified notion of regularity in one hypercomplex variable},  J. Geom. Phys. 202 (2024), Paper No. 105219, 13 pp.
    \bibitem{Ghiloni-Stoppato-24-2} R. Ghiloni, C. Stoppato, \textit{A unified theory of regular functions of a hypercomplex variable}, arXiv:2408.01523, 2024.

  \bibitem{Gurlebeck} K. G\"{u}rlebeck, K. Habetha, W. Spr\"{o}{\ss}ig, \textit{Holomorphic functions in the plane and $n$-dimensional space}, Birkh\"{a}user Verlag, Basel, 2008.

\bibitem{Huo-24} Q. Huo, P. Lian, J. Si, Z. Xu,  \textit{Almansi-type decomposition and Fueter-Sce theorem for
generalized partial-slice regular functions}, arXiv:2411.05571, 2024.

\bibitem{Jin-Ren-20}  M. Jin, G. Ren,   \textit{Cauchy kernel of slice Dirac operator in octonions with complex spine}, Complex Anal. Oper. Theory 14 (2020), no. 1, Paper No. 17, 24 pp.
\bibitem{Jin-Ren-21}  M. Jin, G. Ren,   \textit{ Global Plemelj formula of slice Dirac operator in octonions with complex spine}, Complex Anal. Oper. Theory 15 (2021), no. 2, Paper No. 33, 18 pp.

\bibitem{Jin-Ren-Sabadini-20}  M. Jin, G. Ren, I. Sabadini,  \textit{Slice Dirac operator over octonions}, Israel J. Math. 240 (2020), no. 1, 315-344.

\bibitem{krantzparks} S. G. Krantz, H. R. Parks, \textit{A primer of real analytic functions}, Second ed. Birkh\"auser Advanced Texts, Birkh\"auser, Boston, 2002.


\bibitem{Li-Peng-01}X. Li,  L. Peng, \textit{Taylor series and orthogonality of the octonion analytic functions}, Acta Math. Sci. Ser. B 21 (2001), no. 3, 323-330.
\bibitem{Li-Peng-02}X. Li,  L. Peng, \textit{The Cauchy integral formulas on the octonions}, Bull. Belg. Math. Soc. Simon Stevin 9 (2002), no. 1, 47-64.
\bibitem{Liao-Li-11}J. Liao,  X. Li, \textit{An improvement of the octonionic Taylor type theorem}, Acta Math. Sci. Ser. B  31 (2011), no. 2, 561-566.


 \bibitem{Nono}K. Nono, \textit{On the octonionic linearization of Laplacian and octonionic function theory}, Bull. Fukuoka Univ.
 Ed. Part III, 37(1988), 1-15.
  \bibitem{Okubo} S. Okubo, \textit{
Introduction to octonion and other non-associative algebras in physics},
Montroll Memorial Lecture Series in Mathematical Physics, 2. Cambridge University Press, Cambridge, 1995.

  \bibitem{Perotti-22} A. Perotti, \textit{Cauchy-Riemann operators and local slice analysis over real alternative algebras}, J. Math. Anal. Appl. 516 (2022), no. 1, Paper No. 126480, 34 pp.

  \bibitem{Schafer} R. D. Schafer, \textit{An introduction to nonassociative algebras}, Pure and Applied Mathematics, Vol. 22. Academic Press, New York-London, 1966.
\bibitem{Sommen-81}F. Sommen, \textit{Spherical monogenic functions and analytic functionals on the unit sphere}, Tokyo J. Math. 4 (1981), no. 2, 427-456.
\bibitem{Wang-17}  X. Wang,   \textit{On geometric aspects of quaternionic and octonionic slice regular functions}, J. Geom. Anal. 27 (2017), no. 4, 2817-2871.

\bibitem{Xu-21} Z. Xu, \textit{Bohr theorems for slice regular functions over octonions}, Proc. Roy. Soc. Edinburgh Sect. A 151 (2021), no. 5, 1595-1610.

   \bibitem{Xu-Sabadini} Z. Xu, I. Sabadini, \textit{Generalized partial-slice monogenic functions},  Trans. Amer. Math. Soc. 378 (2025), no. 2, 851-883.

 \bibitem{Xu-Sabadini-2} Z. Xu, I. Sabadini, \textit{On the Fueter-Sce theorem for generalized partial-slice monogenic functions},  Ann. Mat. Pura Appl.  204 (2025), no. 2, 835-857.

\bibitem{Xu-Sabadini-3} Z. Xu, I. Sabadini, \textit{Generalized partial-slice monogenic functions: a synthesis of two function  theories}, Adv. Appl. Clifford Algebr. 34 (2024), no. 2, Paper No. 10.

 \bibitem{Xu-Sabadini-4} Z. Xu, I. Sabadini, \textit{Segal-Bargmann transform for generalized partial-slice
monogenic functions},     Izv. Math. 89 (2025), no. 6, 1182-1207.


\end{thebibliography}

\vskip 10mm
\end{document}